\documentclass{amsart}
\usepackage{graphicx} 
\usepackage{amsmath,amsthm,amssymb} 
\usepackage{mathtools}
\usepackage{mathrsfs}
\usepackage{stmaryrd}
\usepackage{quiver}
\usepackage{enumerate} 
\usepackage[shortlabels]{enumitem}
\usepackage{xcolor}
\usepackage{parskip}
\usepackage{comment}
\usepackage{tikz-cd}
\usepackage[left=0.8in, right=0.8in, top=1.3in, bottom=1.3in]{geometry}
\usepackage{hyperref}
\usepackage{bookmark}
\hypersetup{colorlinks=true, citecolor=blue, urlcolor=blue}
\usepackage[
backend=biber, style=alphabetic, backref=true, backrefstyle=none, minbibnames=6, maxbibnames=10
]{biblatex}
\bookmarksetup{depth=2}
\newtheorem{innercustomthm}{Theorem}
\newenvironment{customthm}[1]
{\renewcommand\theinnercustomthm{#1}\innercustomthm}
{\endinnercustomthm}
\newtheorem*{theorem*}{Theorem}
\newtheorem{theorem}{Theorem}[section]
\newtheorem{lemma}[theorem]{Lemma}
\newtheorem{prop}[theorem]{Proposition}
\newtheorem{cor}[theorem]{Corollary}

\theoremstyle{definition}

\newtheorem{defn}[theorem]{Definition}

\newtheorem{ass}[theorem]{Assumption}

\theoremstyle{remark}

\newtheorem{rem}[theorem]{Remark}

\newcommand{\cal}{\mathcal}

\newcommand{\A}[2]{\mathbb{A}_{#1}^{#2}}
\newcommand{\Q}{\mathbb{Q}}
\newcommand{\Z}{\mathbb{Z}}

\newcommand{\C}{\mathbb{C}}

\newcommand{\F}[1]{\mathbb{F}_{#1}}
\newcommand{\GL}{\operatorname{GL}}

\newcommand{\Gm}{\mathbb{G}_m}
\newcommand{\Sh}[1]{\mathcal{#1}}
\newcommand{\sh}[1]{\mathscr{#1}}
\newcommand{\pp}{\mathfrak{p}}
\newcommand{\Limp}{L^{\mathrm{imp}}_{\mathrm{As}}}
\newcommand{\RG}[1]{\mathrm{R}\Gamma_{#1}}
\newcommand{\et}{\mathrm{\acute{e}t}}
\newcommand{\proet}{\mathrm{pro\acute{e}t}}
\newcommand{\AF}{\mathrm{AF}^{j}_{\et}}
\newcommand{\CG}{\mathrm{CG}}
\newcommand{\rig}{\mathrm{rig}}
\newcommand{\syn}{\mathrm{syn}}
\newcommand{\Eis}{\Sh{E}is}
\newcommand{\Iw}{\mathrm{Iw}}
\newcommand{\st}{\mathrm{st}}
\newcommand{\calF}{\mathcal{F}}
\newcommand{\calG}{\mathcal{G}}
\newcommand{\BGG}{\mathrm{BGG}}

\DeclareMathOperator{\Spec}{Spec}

\DeclareMathOperator{\Res}{Res}
\DeclareMathOperator{\Sym}{Sym}
\DeclareMathOperator{\Fil}{Fil}
\DeclareMathOperator{\Gr}{Gr}

\DeclareMathOperator{\can}{can}

\DeclareMathOperator{\id}{id}
\DeclareMathOperator{\dR}{dR}
\DeclareMathOperator{\HK}{HK}
\DeclareMathOperator{\coh}{coh}
\DeclareMathOperator{\BK}{BK}

\DeclareMathOperator{\Ext}{Ext}
\DeclareMathOperator{\Tr}{Tr}
\DeclareMathOperator{\As}{As}
\DeclareMathOperator{\St}{St}
\DeclareMathOperator{\Cone}{Cone}
\DeclareMathOperator{\ord}{ord}
\DeclareMathOperator{\cris}{cris}
\DeclareMathOperator{\MF}{MF}
\DeclareMathOperator{\Gal}{Gal}

\begin{document}

\title{New perspectives on $p$-adic regulator formulae}

\author{Ting-Han Huang}
\author{Ananyo Kazi}
\author{Luca Marannino}

\date{\today}

\begin{abstract}
Inspired by the pullback method in the recent work of Sangiovanni-Vincentelli--Skinner, we reconstruct the diagonal class of Darmon--Rotger.
Moreover, we reinterpret the computation of the $p$-adic regulator formula.

\end{abstract}

\maketitle

\numberwithin{equation}{section}
\setcounter{tocdepth}{1}
\tableofcontents

\section{Introduction}
Over the last decade, Euler systems (or, more generally, $p$-adic families of cohomology classes) have been used to obtain instances of the Bloch--Kato conjecture for $p$-adic Galois representations of automorphic origin. A typical way to obtain results for the Bloch--Kato conjecture in cases of rank $0$ consists in relating the relevant $L$-value to the Bloch--Kato exponential of the (localisation at $p$ of the) bottom class of an Euler system. Combining this relation with the power of Euler systems, one obtains the proof that the non-vanishing of the $L$-value implies the finiteness of the Bloch--Kato Selmer group.

A successful strategy to obtain this relation -- first appearing in the celebrated work \cite{BDP} -- consists of the following three steps:
\begin{itemize}
    \item construct a $p$-adic $L$-function having the relevant $L$-value in its range of interpolation;
    \item prove a $p$-adic regulator formula relating crystalline classes with values of the $p$-adic $L$-function \emph{outside} the interpolation range;
    \item deform the Euler system and the regulator formula $p$-adically to obtain the desired identity \emph{within} the interpolation range.
\end{itemize}

In this work, we focus on the second step, with the aim of developing new tools and simplifying the existing machinery. We exhibit our techniques in the case of \emph{diagonal classes/triple product $p$-adic $L$-functions}.
Let us clarify from the outset that the novelty of this article does not lie in the results, but rather in their proofs. 
We believe that some of the ideas could be easily adapted to cover other known (or less known) examples. 

\subsection{A brief history}
As mentioned in the beginning, the Galois representations of interest are attached to automorphic forms that moreover can be realised in the cohomology of Shimura varieties. The crystalline classes, whose Bloch--Kato logarithms one intends to interpret via $p$-adic $L$-values typically arise from pushforward of Galois invariant classes along sub-Shimura varieties. The pushforward construction is also motivated by the fact that in such cases, the critical $L$-values can be interpreted as Ichino--Ikeda or Rankin--Selberg type period integrals over spherical pairs $(G,H)$ of reductive groups. The following are instances of $p$-adic regulator formulae via pushforward.
\begin{itemize}
    \item \cite{BDP} (generalised Heegner cycles) -- for $\Res_{K/\Q}\Gm \xhookrightarrow{} \GL_2$ with $K$ quadratic imaginary.
    \item \cite{BDR1}, \cite{KLZ1} (Beilinson--Flach classes) -- for $\GL_2 \xhookrightarrow{} \GL_2 \times \GL_2$.
    \item \cite{DR_diagonal_1}, \cite{BSV_reciprocity_balanced} (diagonal classes) -- for $\GL_2 \xhookrightarrow{} \GL_2 \times \GL_2 \times \GL_2$.
    \item \cite{loeffler2024blochkatoconjecturegsp4} (Lemma--Flach classes) -- for $\GL_2 \times_{\Gm} \GL_2 \xhookrightarrow{} \mathrm{GSp}_4$.
    \item \cite{grossi2025asaiflachclassespadiclfunctions} (Asai--Flach classes) -- for $\GL_2 \xhookrightarrow{} \Res_{F/\Q}\GL_{2}$ with $F$ quadratic totally real.    
\end{itemize}
The key tool in proving the $p$-adic regulator formula in these cases is a generalisation of Coleman's integration theory due to Besser \cite{Besser_inventiones} using syntomic cohomology. Let us briefly explain what role syntomic cohomology plays in these proofs.
\begin{enumerate}
    \item The Galois cohomology classes are constructed via generalisations of the Abel--Jacobi map. For a spherical pair $(G,H)$, corresponding to a Clebsch--Gordan map $V_H \to {V_G}_{|H}$ of irreducible representations, one considers pushforward in absolute \'etale cohomology
    \[
    H^i_{\et}(Y_H, V_H) \to H^{i+2d}_{\et}(Y_G, V_G(d)),
    \]
    where $Y_H \xhookrightarrow{} Y_G$ are Shimura varieties for $H$ and $G$ with the embedding being of codimension $d$. For some special class $\xi \in H^i_{\et}(Y_H, V_H)$ which is usually motivic, if one can show that the image of $\xi$ vanishes under the projection to $H^0(\Q, H^{i+2d}_{\et}(Y_{G,\bar{\Q}}, V_G(d)))$, then the edge map of Hochschild--Serre spectral sequence produces a class $\mathrm{AJ}(\xi) \in H^1(\Q, H^{i+2d-1}_{\et}(Y_{G,\bar{\Q}}, V_G(d)))$. Syntomic cohomology is used to prove that the local class at $p$, denoted by $\mathrm{AJ}_p(\xi)$, lies in the Bloch--Kato finite Selmer group $H^1_f(\Q_p, -)$. This is achieved by a deep theorem (due to Nizio\l{}, Nekov\'{a}\v{r}--Nizio\l{}, etc.) that shows that the syntomic descent spectral sequence is compatible with the Hochschild--Serre spectral sequence.
    \item Once it is proved that $\mathrm{AJ}_p(\xi) \in H^1_f$, one can study its image under the (twisted) Bloch--Kato logarithm map $\widetilde{\log}_{\BK}\colon H^1_f(\Q_p, V) \to D_{\dR}(V)/(1-\varphi)\Fil^0D_{\dR}(V)$ (assuming $V$ is crystalline for simplicity). Under the assumption that $(1-\varphi)$ is invertible, one gets an element in $D_{\dR}(V)/\Fil^0D_{\dR}(V)$, which we can call $\log_{\BK}(\mathrm{AJ}_p(\xi))$. The pairing of this element with a carefully chosen de Rham class in $\Fil^1D_{\dR}(V^*)$ constructed from relevant modular forms is what one wants to relate to a $p$-adic $L$-value. The generalisation of syntomic cohomology to finite polynomial (fp) cohomology due to Besser allows one to express this pairing in terms of cup products of rigid/coherent classes attached to overconvergent modular forms. Let us recall that
    classes in fp-cohomology can be interpreted as generalised Coleman's $p$-adic integrals, and there is an explicit formula for cup products of fp-cohomology classes. These two facts are crucial for proving the regulator formula.
\end{enumerate}

Recently, Marco Sangiovanni Vincentelli and Christopher Skinner \cite{SV-S} has proposed a new method of constructing Euler system classes. Their method relies on their fascinating construction of Eisenstein classes on orthogonal Shimura varieties. Using this they construct Euler system classes for the adjoint motive of an elliptic newform. Notably, their construction of the cohomology classes is ``dual" to the pushforward picture painted above. Indeed, they rely on pulling back their Eisenstein class along the embedding of Shimura varieties for $\mathrm{SO}(2,2) \xhookrightarrow{} \mathrm{SO}(3,2)$, and the extension classes naturally arise from the relative cohomology sequence associated with this embedding. The proof of the regulator formula is notably simpler owing to the pullback nature of the construction -- it boils down to computing the Coleman primitive of their Eisenstein series and pulling it back to the $\mathrm{SO}(2,2)$ Shimura variety.

\subsection{Main results}
Diagonal classes were originally defined in \cite{DR_diagonal_1} by studying Abel--Jacobi images of diagonal cycles inside Kuga--Sato varieties. Later in \cite{BSV_reciprocity_balanced} the authors gave a neat definition by pushing forward a Galois invariant class from the cohomology of a modular curve $Y$ to the cohomology of $Y^3$ with coefficients in certain algebraic representations of $\GL_2$. Indeed, for $f,g,h$ a self-dual triple of modular forms of weight $k+2,l+2,m+2$ such that the triple of weights form the sides of a triangle (i.e., they are \emph{balanced}), one can consider the Clebsch--Gordan map $\Q_p \to V^k \otimes V^l \otimes V^m(\frac{k+l+m}{2})$, where $V^k = \Sym^k\St^{\vee}$ is the symmetric power of the dual of the standard 2-dimensional representation of $\GL_2$. Then this Clebsch--Gordan map induces a pushforward map on \'etale cohomology
\[
H^0_{\et}(Y, \Q_p) \to H^4_{\et}(Y^3, V^k \boxtimes V^l \boxtimes V^m(\tfrac{k+l+m+4}{2})).
\]
For brevity, write $r = \frac{k+l+m}{2}$ (which is an integer, thanks to the self-duality assumption), $V^{[k,l,m]} = V^k \boxtimes V^l \boxtimes V^m$. The diagonal class is then constructed by pushing forward the trivial class in $H^0_{\et}(Y, \Q_p)$ and taking its image under
\[
H^0_{\et}(Y, \Q_p) \to H^4_{\et}(Y^3, V^{[k,l,m]}(r+2)) \xrightarrow{\mathrm{HS}} H^1(\Q, H^3_{\et}(Y^3_{\bar{\Q}}, V^{[k,l,m]}(r+2)) \to H^1(\Q, V^*_{fgh}(-r-1))
\]
where $V^*_{fgh}$ is the tensor product of the dual Deligne Galois representations attached to $f,g,h$ that appears naturally as a quotient of the geometric \'etale cohomology group above. The second map in the above sequence is the edge map in the Hochschild--Serre spectral sequence.

Inspired by the method of \cite{SV-S}, we give a pullback reconstruction of the diagonal class. Let $E_{k(f,g,h)}$ be the extension class 
\[
0 \to V^*_{fgh}(-r-1) \to E_{k(f,g,h)} \to \Q_p \to 0
\]
defining the diagonal class. View the diagonal class as the equivalent extension
\[
0 \to V^*(f)(-r-1) \to \tilde{E}_{k(f,g,h)} \to V(g) \otimes V(h) \to 0.
\]
\begin{customthm}{A}[Proposition \ref{proposition: diagonal pullback}]
\label{thm B - intro}
    The extension class $\tilde{E}_{k(f,g,h)}$ appears in the mapping fibre long exact sequence of the Clebsch--Gordan pullback $\RG{\et,c}(Y^2_{\bar{\Q}}, V^{[l,m]}) \xrightarrow{\CG^*_{r-k}} \RG{\et}(Y_{\bar{\Q}}, V^k(k-r))$ as follows.
    \[\begin{tikzcd}
	0 & {H^1_{\et}(Y_{\bar{\Q}}, V^k(k-r))} & {H^2(\MF(\CG^*_{r-k}))} & {H^2_{\et,c}(Y^2_{\bar{\Q}}, V^{[l,m]})} & 0 \\
	0 & {V^*(f)(-r-1)} & {\tilde{E}_{k(f,g,h)}} & {V(g) \otimes V(h)} & 0
	\arrow[from=1-1, to=1-2]
	\arrow[from=1-2, to=1-3]
	\arrow[two heads, from=1-2, to=2-2]
	\arrow[from=1-3, to=1-4]
	\arrow[from=1-4, to=1-5]
	\arrow[from=2-1, to=2-2]
	\arrow[from=2-2, to=2-3]
	\arrow[from=2-3, to=2-4]
	\arrow[hook, from=2-4, to=1-4]
	\arrow[from=2-4, to=2-5]
\end{tikzcd}\]
    That is, we pullback the top row by the inclusion on the right and pushout by the surjection on the left to construct the diagonal class as an extension class.
\end{customthm}
With this interpretation, the proof of the regulator formula relating the central critical value of the $f$-unbalanced $p$-adic $L$-function to the logarithm of the diagonal class becomes more transparent. We take a rigid primitive over the dagger ordinary locus $X^{\ord,2,\dagger}$ of the differential 2-form associated with $g \times h$ (after annihilating the class with a suitable Frobenius polynomial), pull it back to $X^{\ord}$, apply the Clebsch--Gordan projection, and pair it with the \emph{anti-holomorphic} rigid differential associated with $f$. This, by definition of the $p$-adic $L$-function, coincides with the $p$-adic $L$-value. We refer to Theorem \ref{theorem: diagonal regulator} for a precise statement.

\begin{rem}
    We point out that in his upcoming PhD thesis, Fernando Trejos Su\'arez will provide a pullback construction of a full cyclotomic Euler system of diagonal classes.
\end{rem}

\begin{rem}
In this work, we focused on diagonal classes and triple product $p$-adic $L$-functions. It is worth mentioning that our techniques can be applied to obtain a pullback interpretation of Hirzebruch--Zagier type classes (replacing $Y^2$ with a Hilbert modular surface and $(g,h)$ with a Hilbert cusp form) or Beilinson--Flach classes (replacing $g$ or $h$ with a suitable Eisenstein series).
\end{rem}

\subsection{Organization of the paper}
The material is organized as follows. Section \ref{section-homalg} introduces the tools from homological algebra that allow to reinterpret the construction of diagonal classes. Section \ref{section: onBKlog} gives abstract formulae for the Bloch--Kato logarithm of crystalline classes. Section \ref{S4} introduces the notation for modular curves and explains the pullback reconstruction.
Section \ref{S5} contains the proof of the regulator formula for diagonal classes, cf. Theorem \ref{theorem: diagonal regulator}. Finally, Section \ref{S6} reinterprets the computations of Section \ref{S5} in terms of syntomic complexes.

\subsection{Acknowledgements}
We thank David Loeffler for inviting  T.-H.H. and L.M. to Brig in March 2025, as well as for his constant and valuable support throughout the preparation of this article.
We thank Andrew Graham and Marco Seveso for helpful discussions on a first draft of this work.

T.-H.H. is supported by the grant G5392391GRA3C565LAGAXR from the ANR-DFG Project HEGAL as a postdoctoral researcher in LAGA, Universit\'{e} Paris XIII.
A.K. is supported by the European Research Council through the Horizon 2020 Excellent Science programme (Consolidator Grant ``Shimura varieties and the BSD conjecture'', grant ID 101001051) as a postdoctoral researcher at UniDistance Suisse.
L.M. was funded by the Simons Collaboration on Perfection in Algebra, Geometry and Topology as postdoctoral researcher at CNRS when most of this work was carried out. He is currently funded by the Progetto di Eccellenza U-GOV DECC23\_012\_AC as a postdoctoral researcher at Università degli Studi di Milano.

\section{Some homological algebra}
\label{section-homalg}

We start with an abstract result which boils down to a careful study of edge maps in Grothendieck spectral sequences.

Let $\cal{A}, \cal{B}, \cal{C}$ be abelian categories with enough injectives, together with additive left exact functors $G \colon \cal{A} \to \cal{B}$ and $F \colon \cal{B} \to \cal{C}$ such that $G$ sends injective objects of $\cal{A}$ to $F$-acyclic objects of $\cal{B}$. Let $X, Y$ be bounded cochain complexes in $\cal{A}$, and $f \colon X \to Y$ be a map of complexes. 

Let $x \in R^q(FG)(X)$ for some $q > 0$, and let $y$ be the image of $x$ under the map $FG(f) \colon R^q(FG)(X) \to R^q(FG)(Y)$.
Assume that the image of $y$ under the edge map $R^q(FG)(Y) \to F(R^qG(Y))$ is $0$, i.e., $y \in \Fil^1R^q(FG)(Y)$. Letting $Z := \ker((R^qG(X)) \to (R^qG(Y)))$, this implies that the image of $x$, denoted $\tilde{x}$, under the edge map $R^q(FG)(X) \to F(R^qG(X))$ belongs to $F(Z)$. Let $\mathbf{y}$ be the image of $y$ under the following map.
\[
\Fil^1R^q(FG)(Y) \to R^1F(R^{q-1}G(Y)) \to R^1F\left(\frac{R^{q-1}G(Y)}{R^{q-1}G(X)} \right).
\]

Consider now the distinguished triangle 
\[
RG(Y)[-1] \to \Cone{RG(f)}[-1] \to RG(X) \xrightarrow{+1}
\]
and the following short exact sequence obtained from the long exact sequence of cohomology 
\[
0 \to \frac{R^{q-1}G(Y)}{R^{q-1}G(X)} \to R^{q-1}(\Cone{RG(f))} \to Z \to 0.
\]

\begin{lemma}\label{L201}
    The image of $\tilde{x} \in F(Z)$ under the connecting map of the long exact sequence
    \[
    0 \to F\left( \frac{R^{q-1}G(Y)}{R^{q-1}G(X)} \right) \to F\left( R^{q-1}(\Cone{RG(f))} \right) \to F(Z) \to R^1F\left(\frac{R^{q-1}G(Y)}{R^{q-1}G(X)}\right) \to \cdots
    \]
    coincides with $\mathbf{y}$.
\end{lemma}

\begin{proof}
    \cite[Lemma 9.5]{Jannsen_mixed_motives}.
\end{proof}

\begin{lemma}\label{L202}
    Let $D(\Sh{A})$ be the derived category of an abelian category $\Sh{A}$ as above, and consider the following diagram in $D(\Sh{A})$
    \begin{equation}\label{E206}
    \begin{tikzcd}[column sep = scriptsize]
	X & Y & Z & {} \\
	{X'} & {Y'} & {Z'} & {} \\
	{X''} & {Y''} & {Z''} & {} \\
	{} & {} & {}
	\arrow[from=1-1, to=1-2]
	\arrow[from=1-1, to=2-1]
	\arrow[from=1-2, to=1-3]
	\arrow[from=1-2, to=2-2]
	\arrow["{+1}", from=1-3, to=1-4]
	\arrow[from=1-3, to=2-3]
	\arrow[from=2-1, to=2-2]
	\arrow[from=2-1, to=3-1]
	\arrow[from=2-2, to=2-3]
	\arrow[from=2-2, to=3-2]
	\arrow["{+1}", from=2-3, to=2-4]
	\arrow[from=2-3, to=3-3]
	\arrow[from=3-1, to=3-2]
	\arrow["{+1}"', from=3-1, to=4-1]
	\arrow[from=3-2, to=3-3]
	\arrow["{+1}"', from=3-2, to=4-2]
	\arrow["{+1}", from=3-3, to=3-4]
	\arrow["{+1}"', from=3-3, to=4-3]
    \end{tikzcd} 
    \end{equation}
    that is constructed out of the top left square by following the description in \cite[\href{https://stacks.math.columbia.edu/tag/05R0}{Tag 05R0}]{stacks-project}. Consider the following situation after taking cohomology, where we assume that $H^{n-1}(Z'') = 0$.
    \begin{equation}\label{E207}
    \begin{tikzcd}
	&&& 0 \\
	&& {H^n(Y)} & {H^n(Z)} & {H^{n+1}(X)} \\
	&& {H^n(Y')} & {H^n(Z')} \\
	0 & {H^n(X'')} & {H^n(Y'')} & {H^n(Z'')} \\
	&& {H^{n+1}(Y)}
	\arrow[from=1-4, to=2-4]
	\arrow[from=2-3, to=2-4]
	\arrow[from=2-3, to=3-3]
	\arrow[from=2-4, to=2-5]
	\arrow[from=2-4, to=3-4]
	\arrow[from=3-3, to=3-4]
	\arrow[from=3-3, to=4-3]
	\arrow[from=3-4, to=4-4]
	\arrow[from=4-1, to=4-2]
	\arrow[from=4-2, to=4-3]
	\arrow[from=4-3, to=4-4]
	\arrow[from=4-3, to=5-3, "d^n_{Y''}"]
    \end{tikzcd}
    \end{equation}
    Let $\delta \colon H^n(X'') \cap \ker d^n_{Y''} \to H^{n+1}(X)$ be the connecting map coming from the snake lemma applied to the above situation. Let $\partial \colon H^n(X'') \to H^{n+1}(X)$ be the differential associated with the triangle 
    \[
    X \to X' \to X'' \xrightarrow{\partial} +1.
    \]
    Then $\delta = -\partial$.
\end{lemma}

\begin{rem}\label{rem2.3}
\begin{enumerate}
    \item We note that the $3 \times 3$ diagram above is not functorial in any of the squares as homotopy limits are not functorial in $D(\Sh{A})$.
    \item The lemma applies to a ``short exact sequence of short exact sequences" where it proves a variant of \cite[Lemma 2.8]{Srinivas}.
\end{enumerate}
\end{rem}

\begin{proof}[Proof of Lemma \ref{L202}]
    We begin by recalling the proof of the fact that such a $3 \times 3$ diagram (\ref{E206}) can be constructed (non-canonically) out of a commutative square in $D(\Sh{A})$. Indeed, the main point of the proof, which can be found in \cite[\href{https://stacks.math.columbia.edu/tag/05R0}{Tag 05R0}]{stacks-project} is the following.

    Let $X \to Y' \to A \to +1$ be a distinguished triangle. Then one can show that there exists a diagram
    \begin{equation}
        \begin{tikzcd}
        \coordinate (A) at (0:1.8cm);    
        \coordinate (B) at (60:1.8cm);
        \coordinate (C) at (-60:1.8cm);
        \coordinate (D) at (180:1.8cm);
        \coordinate (E) at (-120:1.8cm);
        \coordinate (F) at (120:1.8cm);
        \coordinate (G) at (0:0cm);

        \node (1) at (A) {Z'};
        \node (2) at (B) {Z};
        \node (3) at (C) {X[1]};
        \node (4) at (D) {X''};
        \node (5) at (E) {Y''};
        \node (6) at (F) {Y'};
        \node (7) at (G) {A};

        \arrow[from=7, to=1, "{z'}"']
        \arrow[from=2, to=7, "z"]
        \arrow[from=7, to=3, "x"']
        \arrow[from=4, to=7, "{x''}"]
        \arrow[from=7, to=5, "{y''}"']
        \arrow[from=6, to=7, "{y'}"]
        \end{tikzcd}
    \end{equation}
    where the diameters form distinguished triangles. In particular, $xy' = z'x'' = y''z = 0$. Furthermore, the compositions
    \[
    X'' \xrightarrow{y''x''} Y'', \qquad X'' \xrightarrow{xx'' = \partial} X[1], \qquad Y' \xrightarrow{y''y'} Y'', \qquad Y' \xrightarrow{z'y'} Z', \qquad Z \xrightarrow{z'z} Z', \qquad Z \xrightarrow{xz} X[1]
    \]
    define the respective arrows appearing in (\ref{E206}). The fact that the rows and columns of (\ref{E206}) form distinguished triangles then follows from a standard application of the \emph{octahedral axiom}. Let us now prove the main claim of the lemma.

    Let $\eta \in H^n(X'') \cap \ker d^n_{Y''}$. Then there exists $\eta' \in H^n(Y')$ whose image in $H^n(Y'')$ coincides with the image of $\eta$, i.e. 
    \begin{equation}
        y''\left( x''(\eta) - y'(\eta')\right) = 0.
    \end{equation}
    As a result, there exists $\eta'' \in H^n(Z)$ such that 
    \begin{equation}
        x''(\eta) - y'(\eta') = -z(\eta'').
    \end{equation}
    This shows that $z'z(\eta'') = z'y'(\eta')$, meaning that $\eta''$ is the unique lift to $H^n(Z)$ of the image of $\eta'$ in $H^n(Z')$. Therefore
    \begin{equation}
        \delta(\eta) = xz(\eta'') = -xx''(\eta) = -\partial(\eta)
    \end{equation}
    which completes the proof.
\end{proof}

\section{On the Bloch-Kato logarithm of an extension}
\label{section: onBKlog}

In this section we review several recipes for the computation of the Bloch--Kato logarithm.
The main references are \cite{Nekovarheight} and \cite[\S 7.2]{SV-S}.
As we expect the reader to be already familiar with $p$-adic Hodge theory, we will directly refer to the notes \cite{BC2009} for detailed definitions of the relevant period rings and the category of filtered $(\varphi, N)$-modules.

Let $K$ be a finite extension of $\Q_p$ and $K_0$ be its maximal unramified subfield.
We let $\varphi$ be the absolute Frobenius on $K_0$ and let $G_K = \mathrm{Gal}(\overline{K}/K)$.
We let $t\in B_{\dR}$ denote Fontaine's $2\pi i$, which depends on a choice of compatible $p$-th power roots of unity.

In this section we will only consider $\Q_p$-linear representations, but everything can be generalized to the $L$-linear setting for $L/\Q_p$ a finite (large enough) extension. We will need this generalization for the applications.



Let $Rep_{\st}(G_K)$ be the category of semistable representations and $MF_K(\varphi, N)$ be the category of filtered $(\varphi, N)$-modules over $K$ \cite[\S 1.6]{Nekovarheight}.
Then $D_{\st}$ can be viewed as a functor 
\[ D_{\st} : Rep_{\st}(G_K) \rightarrow MF_K(\varphi, N)\]
whose essential image are the admissible filtered $(\varphi, N)$-modules (denoted by $MF_K^{ad}(\varphi, N)$).

We let $e_n\in D_{\cris}(\Q_p(n))$ denote a canonical generator. It can be obtained as $e_n=t^{-n}\otimes v_n$, where $v_n\in\Z_p(n)$ is the generator coming from the same choice of $p$-th power roots of unity appearing in the definition of $t\in B_{\cris}\subset B_{\dR}$.

Let $V$ be a $p$-adic representation. We define for $\ast \in \{ \dR, \cris, \st \}$
\[ H^1_{\ast}(K, V) := \ker \left ( H^1(K, V) \rightarrow H^1(K, V \otimes_{\Q_p} B_\ast) \right ). \]
If $V$ is a $\ast$-representation, then $H^1_{\ast}(K, V)$ classifies extensions $0 \rightarrow V \rightarrow E \rightarrow \Q_p(0) \rightarrow 0$ in which $E$ is also a $\ast$-representation \cite[Proposition 1.26]{Nekovarheight}.
We shall adopt the common notations $H^1_g:= H^1_{\dR}, H^1_f:= H^1_{\cris}$ and  
\[ H^1_e(K, V) := \ker \left ( H^1(K, V) \rightarrow H^1(K, V \otimes_{\Q_p} B_{\cris}^{\varphi=1}) \right ). \]


Let $V$ be a de Rham representation. 
We define two complexes (of $K_0$-vector spaces)
\begin{align*}
    C^\bullet_{\cris}(V) := &[ D_{\cris}(V) \xrightarrow[]{(1 -\varphi,\iota )} D_{\cris}(V) \oplus D_{\dR}(V)/ \Fil^0 ]\\
    C^\bullet_{\st}(V) := &[ D_{\st}(V) \xrightarrow[]{(1 -\varphi, N,\iota )} D_{\st}(V) \oplus D_{\st}(V) \oplus D_{\dR}(V)/ \Fil^0 \xrightarrow[]{(N, -(1- p\varphi), 0)} D_{\st}(V) ] 
\end{align*}
starting at degree $0$. Here, for $\ast\in\{\cris,\st\}$, $\iota$ denotes the natural map $D_{\ast}(V) \rightarrow D_{\dR}(V) \rightarrow D_{\dR}(V)/\Fil^0$.

Then there are natural identifications
\[
H^0(K, V) \cong H^0(C^\bullet_{\cris}(V)) = H^0(C^\bullet_{\st}(V)),\qquad H^1_f(K, V) \cong H^1(C^\bullet_{\cris}(V)),\qquad H^1_{\st}(K, V) \cong H^1(C^\bullet_{\st}(V)).
\]

From now on, we will assume that $V$ is a semistable $G_K$-representation.

\begin{rem}
    We also need the following alternative description of $C^\bullet_{\st}(V)$:
    \[
        \tilde{C}^\bullet_{\st}(V) := [ D_{\st}(V) \oplus \Fil^0 D_{\dR}(V) \rightarrow D_{\st}(V) \oplus D_{\st}(V) \oplus D_{\dR}(V) \xrightarrow[]{(N, -(1 -p\varphi), 0)} D_{\st}(V) ]
    \]
    where the first map is given by $[u, v] \mapsto [(1-\varphi)u, Nu, u-v]$.
    It is clear that the two complexes ${C}^\bullet_{\st}(V)$ and $\tilde{C}^\bullet_{\st}(V)$ compute the same cohomology.
    More generally, for any polynomial $P \in 1+ T \cdot \Q_p[T]$, we can consider the complexes $C^\bullet_{\st, P}(V)$ and $\tilde{C}^\bullet_{\st, P}(V)$ in which the map $1-\varphi$ (resp. $-(1-p\varphi)$) is replaced by $P(\varphi)$ (resp. $-P(p\varphi)$).
    Again, the two complexes $C^\bullet_{\st, P}(V)$ and $\tilde{C}^\bullet_{\st, P}(V)$ compute the same cohomology (which we will denote by $H^\bullet_{\st,P}(K,V)$), and we will use them interchangeably.
    
\end{rem}

Let us now describe the isomorphism $H^1_{\st}(K, V) \cong H^1(C^\bullet_{\st}(V))$.
Given $E \in H^1_{\st}(K, V)$, it can be viewed as an extension of $\Q_p$ by $V$.
Then applying $D_{\st}$ and $D_{\dR}$, we get extensions 
\[ 0 \rightarrow D_{\st}(V) \rightarrow D_{\st}(E) \rightarrow K_0 \rightarrow 0
\]
and 
\[ 0 \rightarrow D_{\dR}(V) = D_{\st}(V) \otimes K \rightarrow D_{\dR}(E) = D_{\st}(E) \otimes K\rightarrow K \rightarrow 0.
\]

Let $s: K_0 \rightarrow D_{\st}(E)$ be a $K_0$-linear section. We extend it to a $K$-linear section $K \rightarrow D_{\dR}(E)$ and still denote the map by $s$.
Consider the following diagram
\begin{equation}
\label{extdiagram-BK}
\begin{tikzcd}
    0 \arrow[r] &D_{\dR}(V) \arrow[r] \arrow[d] &D_{\dR}(E) \arrow[r] \arrow[d, "\iota"] &K \arrow[l, bend right, "s"'] \arrow[r] \arrow[d] &0 \\
    0 \arrow[r] &D_{\dR}(V)/ \Fil^0 \arrow[r, "\sim" ', "\delta"] &D_{\dR}(E) /\Fil^0 \arrow[r] &0 
\end{tikzcd}.    
\end{equation}
Then the triple
\[
[a, b, c] := \left [ (1-\varphi) s(1), Ns(1), \delta^{-1} \circ \iota \circ s(1) \right ] \in D_{\st}(V) \oplus D_{\st}(V) \oplus D_{\dR}(V)/ \Fil^0 
\]
is a $1$-cocycle in $C^\bullet_{\st}(V)$.
One can verify that the class of $[a, b, c]$ is independent of the choice of $s$.
Hence, we get a map $H^1_{\st}(V) \rightarrow H^1(C^\bullet_{\st}(V))$ that one can prove is an isomorphism \cite[\S 1.19]{Nekovarheight}.


\begin{lemma}[{\cite[Lemma 7.4.4]{SV-S}}]
    Assume that $D_{\st}(V)^{\varphi =1}=0$.
    Then the map $H^1_{\st}(K, V) \rightarrow D_{\dR}(V)/ \Fil^0$ induced by the map $C^1_{\st}(V) \rightarrow D_{\dR}(V)/ \Fil^0$
    \[
    [a, b, c] \mapsto c - (1-\varphi)^{-1} (a) \mod \Fil^0 D_{\dR}(V)
    \]
    is well-defined and is the semistable Bloch--Kato logarithm $\log_{\BK,\st}$.    
\end{lemma}

\begin{rem}
    This map is not injective in general. On the other hand, if $V$ is $K$-crystalline such that $1-\varphi$ and $1-p\varphi$ are bijective on $D_{\st}(V)$, then the above map coincides with the usual Bloch--Kato logarithm (on $H^1_e$) under the identification $H^1_e(K,V)=H^1_g(K,V)\cong H^1(C^\bullet_{\st}(V))$
\end{rem}

Summarizing the above result, we have the following proposition (cf. \cite[Proposition 7.4.5]{SV-S}).
\begin{prop}
\label{Prop: BKlog1}
    Let $ 0 \rightarrow V \rightarrow E \rightarrow \Q_p \rightarrow 0$ be an extension of semi-stable representations of $G_K$. Assume that $1-\varphi$ is bijective on $D_\st(V)$. 
    Choose a $K_0$-linear section $s: K_0 \rightarrow D_{\st}(E)$ and extend it $K$-linearly to $s: K \rightarrow D_{\dR}(E)$ as before.

    Then, with the notation of diagram \eqref{extdiagram-BK},
    \[
    \log_{\BK,\st}(E) =\delta^{-1} \circ s(1) - (1-\varphi)^{-1}\big((\delta^{-1} \circ (1-\varphi)s(1)\big) \mod \Fil^0 D_{\dR}(V).
    \]
\end{prop}

\begin{rem}
    In practice, it will be convenient to choose a section $s$ that respects some of the filtered $(\varphi, N)$-module structures.
    We will illustrate two examples below.
    
    When $s$ respects the filtrations, then $\iota \circ s (1) =0$ and $\log_{BK,\st}(E) = - (1-\varphi)^{-1}\big(\delta^{-1}\circ(1-\varphi)s(1)\big)$.
    Equivalently, by the assumption of the Frobenius action, there is a unique $\varphi$-equivariant splitting $t_{\varphi}: D_{\st}(E) \rightarrow D_{\st}(V)$ which we extend to $t_{\varphi}: D_{\dR}(E) \rightarrow D_{\dR}(V)$ linearly.
    Then $t_{\varphi} \circ s(1)$ is well-defined modulo $\Fil^0 D_{\dR}(V)$.
    In fact, it is equal to $(1-\varphi)^{-1} \circ \delta^{-1} \circ (1-\varphi)s(1)$ in $D_{\dR}(V)/ \Fil^0$.
    
    If we can choose $s$ such that it respects the $(\varphi, N)$-structures, then $(1-\varphi)s(1) = 0 = Ns(1)$ and furthermore $H^1_{\st}(K, V) \cong D_{\dR}(V)/ \Fil^0 D_{\dR}(V)$.
    This is in accordance with \cite[Lemma 2.1]{Iovita-Spiess-derivative}.
\end{rem}

We will also need a slighly more sophisticated formula for the Bloch--Kato logarithm in the following setting.

Let
\begin{equation}
\label{E1ext}
0\to V\to E_1\to W\to 0   
\end{equation}
be a $K$-semistable extension of semistable $G_K$-representations. According to \cite[Proposition 1.26]{Nekovarheight}, there is a natural isomorphism
\begin{equation}
\label{extiso}
\Ext^1_{K,\st}(W,V)\xrightarrow{\cong}H^1_{\st}(K,V\otimes_{\Q_p}W^\vee).  
\end{equation}
The image of the class of the extension \eqref{E1ext} under this isomorphism is explicitly described by the bottom row of the pullback diagram
\[
\begin{tikzcd}
            0 \arrow[r] & V\otimes_{\Q_p}W^\vee\arrow[r] &E_1\otimes_{\Q_p}W^\vee \arrow[r]& W\otimes_{\Q_p}W^\vee\arrow[r] &0 \\
            0 \arrow[r] &V\otimes_{\Q_p}W^\vee \arrow[r]\arrow[u, equal] &E_2\arrow[u,hook] \arrow[r] & \Q_p\arrow[u,hook] \arrow[r]  &0
    \end{tikzcd}.
\]

\begin{ass}
From now one we assume that $V$ and $W$ are $K$-crystalline and, moreover, that $1-\varphi$ and $1-p\varphi$ are bijective on $D_{\dR}(V\otimes_{\Q_p}W^\vee)$.   
\end{ass}

Upon fixing an element $\omega\in\Fil^nD_{\dR}(W)\cap D_{\st}(W)$, we obtain a well-defined linear map
\[
ev_\omega:D_{\dR}(V\otimes_{\Q_p}W^\vee)/\Fil^0\to D_{\dR}(V(n))/\Fil^0 D_{\dR}(V(n)) =D_{\dR}(V)/\Fil^n D_{\dR}(V).
\]
In particular, it makes sense to consider $ev_\omega(\log_{\BK}(E_2))$. 

\begin{prop}
\label{Prop: BKlog2}
In the above setting, assume that there exists a polynomial $P(T)\in 1+T\cdot\Q_p[T]$ such that $P(\varphi)(\omega)=0$ and such that $P(\varphi)$ and $P(p\varphi)$ are bijective on $D_\st(V(n))$. Fix any $K_0$-linear section $s_1:D_{\st}(W)\to D_{\st}(E_1)$ and extend it $K$-linearly to $s_1:D_{\dR}(W)\to D_{\dR}(E_1)$. Then
\[
ev_\omega(\log_{\BK}(E_2))=s_1(\omega)-P(\varphi)^{-1}\big(P(\varphi)s_1(\omega)\big)\mod\Fil^n D_{\dR}(V).
\]
\end{prop}

\begin{proof}
The element $\log_{\BK}(E_2)$ is computed via the recipe of Proposition \ref{Prop: BKlog1}. For any $K_0$-linear splitting $s_2\colon K_0=D_{\st}(\Q_p)\to D_{\st}(E_2)$ we have
\[
\log_{\BK}(E_2)=s_2(1)-(1-\varphi)^{-1}\big((1-\varphi_{E_2})s_2(1)\big)\mod\Fil^0 D_{\dR}(V \otimes W^\vee).
\]
Arguing in the same way as in \cite[Proposition 1.4.3]{BLZ2016}, one can compute the cup product $[w,x,y]$ of the classes $[(1-\varphi)s_2(1),N(s_2(1)),s_2(1)]$ and $[\omega] \in H^0_{\st,P}(W(n))$ according to the recipe in \cite[Table 1]{BLZ2016} and then obtain
\[
ev_\omega(\log_{\BK}(E_2))=y-\big(P(\varphi)^{-1}(w)\mod\Fil^0 D_{\dR}(V(n)) \big).
\]
Concerning the cup product, we have to fix polynomials $a(T_1,T_2), b(T_1,T_2)$ such that $a(T_1,T_2)\cdot (1-T_1)+b(T_1,T_2)P(T_2)=P(T_1T_2)$. The obvious choice is $b(T_1,T_2)=1$ and $a(T_1,T_2)=(P(T_1T_2)-P(T_2))/(1-T_1)$. According to the table, we get
\[
[w,x,y]=\text{(image under $\lambda$ of) }[ a(\varphi_1,\varphi_2)((1-\varphi)s_2(1)\otimes\omega),N(s_2(1))\otimes\omega, s_2(1)\otimes\omega ]
\]
where
\[
\lambda:D_{\dR}(V\otimes_{\Q_p}W^\vee)\otimes_{\Q_p}D_{\dR}(W(n))\to D_{\dR}(V(n))
\]
is the obvious pairing. In order to complete the proof we are thus left to show that
\[
[P(\varphi)(s_1(\omega)),N(s_1(\omega)),s_1(\omega)]=[w,x,y]\in H^1_{\st,P}(V(n)).
\]
It is easy to see that the inverse of the isomorphism \eqref{extiso} can be explicitly described via the pushout diagram
\begin{equation*}
    \begin{tikzcd}
            0 \arrow[r] & V\otimes_{\Q_p} W^\vee\otimes_{\Q_p}W\arrow[r]\arrow[d, twoheadrightarrow] &E_2\otimes_{\Q_p}W \arrow[r]\arrow[d, twoheadrightarrow] & W\arrow[r]\arrow[d, equal] &0 \\
            0 \arrow[r] &V \arrow[r] &E_3  \arrow[r] & W \arrow[r]  &0
    \end{tikzcd},
\end{equation*}
i.e., $[E_1]=[E_3]$ in $\Ext^1_{K,\st}(W,V)$.

The above diagram is helpful because now we can apply $D_{\st}(-)$ to it and consider compatible splittings
\begin{equation*}
    \begin{tikzcd}
            0 \arrow[r] & D_{\st}(V\otimes_{\Q_p}W^\vee\otimes_{\Q_p}W)\arrow[r]\arrow[d, twoheadrightarrow,"\lambda"] &D_{\st}(E_2\otimes_{\Q_p}W) \arrow[r]\arrow[d, twoheadrightarrow,"\mu"] & D_{\st}(W)\arrow[l, bend right=30, "s"']\arrow[r]\arrow[d, equal] &0 \\
            0 \arrow[r] &D_{\st}(V) \arrow[r] &D_{\st}(E_1)  \arrow[r]\arrow[u,bend left=50, "t"] & D_{\st}(W) \arrow[l,bend left=30,"s_1"] \arrow[r]  &0
    \end{tikzcd}
\end{equation*}
such that $s=t\circ s_1$, where $t$ is a $K_0$-linear section of the surjection $\mu$ chosen appropriately.

Note that we must have $s(\omega)=(\alpha\otimes\omega)+\beta$ for some $\alpha\in D_{\st}(E_2)$ and some
\[
\beta\in\ker\big(D_{\st}(E_2\otimes_{\Q_p}W)\to D_{\st}(W)\big).
\]
This defines a section $s_2:K_0\to D_{\st}(E_2)$ by declaring $s_2(1)=\alpha$. Now it is clear that by making this choice (which we can make) we would have $s_1(\omega)=\mu(s_2(1)\otimes\omega)-\mu(\beta)$, so that it follows automatically (recall that $N(\omega)=0$ since we assume that $W$ is $K$-crystalline) that $N(s_1(\omega))=x-N(\mu(\beta))$ and $s_1(\omega)=y-\mu(\beta)$. Since
\[
\mu(\beta)\in\ker\big(D_{\st}(E_1)\to D_{\st}(W)\big),
\]
$[P(\varphi(\mu(\beta))),N(\mu(\beta)),\mu(\beta)]=0$ in $H^1_{\st,P}(V(n))$.

Hence, we are left to check that $P(\varphi)(s_1(\omega))=w-P(\varphi)(\mu(\beta))$, which follows directly from the following chain of equalities
\begin{align*}
\lambda \big(a(\varphi_1,\varphi_2)((1-\varphi)(s_2(1))\otimes\omega)\big)&=\mu \Big(\tfrac{P(\varphi_1\otimes\varphi_2)-P(\varphi_2)}{1-\varphi_1}\big((1-\varphi_1)s_2(1)\otimes\omega\big)\Big)=\\
&=\mu\big(P(\varphi_1\otimes\varphi_2)(s_2(1)\otimes\omega)\big)=\\
&=P(\varphi)\big(\mu(s_2(1)\otimes\omega)\big)=P(\varphi)(s_1(\omega))-P(\varphi)(\mu(\beta)).    
\end{align*}
\end{proof}

\section{Pullback reconstruction of diagonal classes}
\label{S4}
\subsection{Modular curves and automorphic coefficients}
For an integer $N >3$, let $Y_1(N)/\Spec{\Q}$ be the modular curve of level $\Gamma_1(N)$. Since we will usually fix the level, we will often denote this modular curve by $Y$ and let $\bar{Y} = Y_{\bar{\Q}}$. 
We denote the smooth compactification of $Y$ by $X$.
The complement of $Y$ in $X$ (the cusps) will be denoted by $D$. We note that if $p \nmid N$, then $Y$ has a smooth model over $\Z_{(p)}$.

There exists a universal elliptic curve $\pi \colon \Sh{E} \to Y$. Let $\omega_{\Sh{E}} = \pi_*\Omega_{\Sh{E}/Y}$ be the sheaf of invariant differentials. This sheaf has a canonical extension to $X$ which we still denote by $\omega_{\Sh{E}}$.
Let $\Sh{H}_{\Sh{E}} := R^1\pi_*\Omega^{\bullet}_{\Sh{E}/Y}$ be the relative de Rham sheaf which also canonically extends to $X$. There exists a Hodge filtration
\[
0 \to \omega_{\Sh{E}} \to \Sh{H}_{\Sh{E}} \to \omega_{\Sh{E}^{\vee}}^{\vee} \to 0.
\]
Using the principal polarisation $\Sh{E} \simeq \Sh{E}^{\vee}$, we identify $\omega_{\Sh{E}} \simeq \omega_{\Sh{E}^{\vee}}$.

For an algebraic representation $V$ of $\GL_{2,\Q}$, there is an associated vector bundle $\Sh{V}$ on $Y$. Our convention for this association is that the dual of the 2-dimensional standard representation $\St^{\vee}$ goes to $\Sh{H}_{\Sh{E}}$. We will often use the notation $V^{k}$ to denote the representation $\Sym^k\St^{\vee}$ and $\Sh{V}^k = \Sh{H}_{\Sh{E}}^k := \Sym^k\Sh{H}_{\Sh{E}}$. If we want to specify the infinity type of the central character we will write $V^{[k;w]}$ for $\Sym^k\St^{\vee} \otimes \det^{\tfrac{k-w}{2}}$, and the corresponding vector bundle is $\Sh{V}^{[k;w]} = \Sh{H}_{\Sh{E}}^{[k;w]}:= \Sym^k\Sh{H}_{\Sh{E}} \otimes (\wedge^2\Sh{H}_{\Sh{E}})^{\tfrac{w-k}{2}}$.

The vector bundle $\Sh{H}_{\Sh{E}}$ is equipped with the integrable Gauss--Manin connection $\nabla \colon \Sh{H}_{\Sh{E}} \to \Sh{H}_{\Sh{E}}\otimes \Omega_{X}(\log D)$ that satisfies Griffiths' tranversality and induces the Kodaira--Spencer map. The Kodaira--Spencer map is an isomorphism $\omega_{\Sh{E}}^2 \otimes (\wedge^2\Sh{H}_{\Sh{E}})^{-1} \simeq \Omega_{X/\Q}(\log D)$. The Gauss--Manin connection extends naturally to algebraic representations, and gives rise to the logarithmic de Rham complex
\begin{align*}
    \RG{\dR}(Y, \Sh{V}) = \RG{\dR}(X, \Sh{V}\langle D \rangle) := \RG{}\left(X, \left[ \Sh{V} \xrightarrow{\nabla} \Sh{V} \otimes \Omega_{X}(\log D)\right]\right).
\end{align*}
We also have the minus log de Rham complex
\begin{align*}
    \RG{\dR,c}(Y, \Sh{V}) = \RG{\dR}(X, \Sh{V}\langle -D \rangle) := \RG{}\left(X, \left[ \Sh{V}(-D) \xrightarrow{\nabla} \Sh{V}(-D) \otimes \Omega_{X}(\log D)\right]\right).
\end{align*}

\subsection{Pullback reconstruction of diagonal classes}
\label{subsection: pullback diagonal classes}
In this section, we briefly recall the definition of diagonal classes as defined in \cite{BSV_reciprocity_balanced}, whose Bloch--Kato logarithm recovers the values of the triple product $p$-adic $L$-function as defined in \cite{DR_diagonal_1} (for Hida families) or \cite{andreatta2021triple} (for Coleman families) in the so-called balanced region. The regulator formula in the Hida family case was already proven in \cite{DR_diagonal_1}, while for the case of Coleman families this was proven in the PhD thesis of the first named author \cite{Huang_tripleL_and_GZ_formula}. Both these results were proven assuming the $\GL_2(\Q_p)$-smooth representations corresponding to the three modular forms are unramified principal series. The PhD thesis of the last named author \cite{marannino2025explicitreciprocitylawsdiagonal} extended these results to allow the non-dominant weight modular forms to be supercuspidal at $p$. In the following we will revisit the regulator formula in the unramified principal series case via our new approach. In particular, we will assume that $p\nmid N$ in what follows.

For a tuple of natural numbers $(k_1, \dots, k_m)$, let $V^{[k_1, \dots, k_m]} := V^{k_1} \boxtimes \cdots \boxtimes V^{k_m}$ as a representation of $\GL_2(\Q_p)^m$ and $\Sh{V}^{[k_1,\dots, k_m]}$ be the corresponding \'etale local system on $Y^m$. 

Let $(k,l,m)$ be a \emph{balanced} triple of natural numbers (i.e., $(k+2,l+2,m+2)$ satisfies the triangle inequality). Let $r = (k+l+m)/2$ and assume $r\in\Z$. Then there is a unique $\GL_2(\Q_p)$-equivariant map \cite[95]{BSV_reciprocity_balanced}
\[
\det{\!}^r \colon  \Q_p \to V^k \otimes V^l \otimes V^m(r).
\]

Let $\Delta \colon Y \to Y^3$ be the diagonal embedding. Since $\Delta^!\Q_p = \Q_p(-2)[-4]$, the map above induces a pushforward map which fits into a distinguished triangle
\begin{equation}\label{E08}
\RG{\proet}(\bar{Y}, \Q_p) \xrightarrow{\Delta_*} \RG{\proet}(\bar{Y}^3, \Sh{V}^{[k,l,m]}(r+2))[4] \to \MF{(\Delta_*)[1]} \xrightarrow{+1}
\end{equation}
where we write $\MF{(\Delta_*)[1]}$ for the unique (up to non-unique isomorphism) object in the derived category making the sequence a distinguished triangle. The corresponding long exact sequence gives rise to a short exact sequence
\[
0 \to H^3_{\proet}(\bar{Y}^3, \Sh{V}^{[k,l,m]}(r+2)) \to H^0(\MF(\Delta_*)) \to H^0_{\proet}(\bar{Y}, \Q_p) \to 0.
\]
The following lemma and the corollary implies that $H^{0}(\MF{}(\Delta_*))$ is a continuous $\Gal(\bar{\Q}/\Q)$-representation.

\begin{lemma}\label{L22}
    Let $S = \Spec{\Q}$. The subcategory $\mathrm{Loc}^{\proet}_S(L) \subset \mathrm{Shv}_{\proet}(S, L)$ of pro\'etale $L$-local systems for any finite extension $L/\Q_p$ is equivalent to the category of continuous $\Gal(\bar{\Q}/\Q)$-representations that are finite dimensional over $L$. Moreover, this is a weak Serre subcategory.
\end{lemma}

\begin{proof}
    An $L$-local system is by definition locally constant for the pro\'etale topology. Hence by \cite[Example 7.3.5]{Bhatt-Scholze} it descends to an \'etale local system. The second claim follows from \cite[Corollary 6.8.5]{Bhatt-Scholze} and \cite[Lemma 18.43.5]{stacks-project}.
\end{proof}

\begin{cor}
    The subcategory $D^+_{\mathrm{Loc}_S^{\proet}(L)}(\mathrm{Shv}_{\proet}(S, L)) \subset D^+_{\proet}(S, L)$ of objects $X \in D^+_{\proet}(S, L)$ satisfying $H^n(X) \in \mathrm{Loc}^{\proet}_S(L)$ for all $n \in \Z$ is a strictly full saturated triangulated subcategory.
\end{cor}

\begin{proof}
    Follows from Lemma \ref{L22} and \cite[Lemma 13.17.1]{stacks-project}.
\end{proof}

Fix three cuspidal eigenforms $f,g,h$ of level $\Gamma_1(N)$ (not necessarily new of level $N$), weights $k+2, l+2, m+2$, and nebentypus characters $\chi_f, \chi_g, \chi_h$ respectively. Assume that the product $\chi_f \cdot \chi_g \cdot \chi_h$ is the trivial Dirichlet character modulo $N$. This should be thought of as a self-duality condition. In particular, it ensures that $r=(k+l+m)/2$ is a non-negative integer.

\begin{defn}
\label{def: Galoisreps modular forms}
For an eigenform $\xi\in S_{\nu}(N,\chi_{\xi}, L)$, we let $V_N(\xi)\subset H^1_{\proet, c}(\bar{Y}, \Sh{V}^k) \otimes L$ be maximal subspace on which the generalised eigenvalue of $T_{\ell}$ for any $\ell \nmid pN$ is given by the trace of the arithmetic Frobenius at $\ell$. 
If $\xi$ is new of level $N$, then $V_N(\xi)$ coincides with the Galois representation attached to $\xi$ by Deligne.
For what follows, the coefficient field $L$ will not play any significant role, and we will assume for simplicity of notation that $L = \Q_p$. We let $H^1_{\proet}(\bar{Y}, \Sh{V}^k(k+1)) \twoheadrightarrow V_N^*(\xi)$ be the dual. When $N$ is clearly from the context, we omit it from the notation.
\end{defn}

Using $H^0_{\proet}(\bar{Y}, \Q_p) = \Q_p$, and pushing out the exact sequence along 
\[
H^3_{\proet}(\bar{Y}, \Sh{V}^{[k,l,m]}(r+2)) \to V^*(f) \otimes V^*(g) \otimes V^*(h)(-r-1) =:V^*_{fgh}(-r-1)
\]
we get an extension class
\begin{equation}\label{E09}
0 \to V^*_{fgh}(-r-1) \to E_{k(f,g,h)} \to \Q_p \to 0
\end{equation}
which by Lemma \ref{L201} coincides with the diagonal class defined in \cite[Section 3]{BSV_reciprocity_balanced}.

Recall the square-root unbalanced triple product $p$-adic $L$-function $\sh{L}_p^{\mathbf{f}}(\mathbf{f}, \mathbf{g}, \mathbf{h})$ associated with a triple of Hida (or Coleman) families. It interpolates central critical $L$-value of $L(V(\mathbf{f}_x) \otimes V(\mathbf{g}_y) \otimes V(\mathbf{h}_z), s)$ in the range where $x \geq y +z$, i.e. the specialisation of $\mathbf{f}$ has the dominant weight. Since we are interested in relating the Bloch--Kato logarithm of $[E_{k(f,g,h)}]$ to the special values of $\sh{L}_p^{\mathbf{f}}(\mathbf{f}, \mathbf{g}, \mathbf{h})$ at a balanced triple $(k, l, m)$, it will be convenient to write the extension class as 
\begin{equation}\label{E10}
0 \to V^*(f)(-r-1) \to \tilde{E}_{k(f,g,h)} \to V(g) \otimes V(h) \to 0.
\end{equation}
We note that $V(g) \otimes V(h) \subset H^2_{\proet,c}(\bar{Y}^2, \Sh{V}^{[l,m]})$, and $H^1_{\proet}(\bar{Y}, \Sh{V}^k(k-r)) \twoheadrightarrow V^*(f)$.

In the following we drop the ``$\proet$" from the notation for simplicity.
\begin{prop}
\label{proposition: diagonal pullback}
    \begin{enumerate}
        \item There is a unique (up to multiplication by a non-zero scalar) $\GL_2(\Q_p)$-equivariant map $\CG_{r-k} \colon  V^l\otimes V^m \to V^k(k-r)$ that induces a map 
        \[
        \RG{c}(\bar{Y}^2, \Sh{V}^{[l,m]}) \xrightarrow{\CG^*_{r-k}} \RG{}(\bar{Y}, \Sh{V}^k(k-r)).
        \]
        \item The exact sequence (\ref{E10}) appears in the long exact sequence associated with the distinguished triangle completing the map above
        \begin{equation}
        \label{disttriang-diagonalpullback}
        \RG{c}(\bar{Y}^2, \Sh{V}^{[l,m]}) \xrightarrow{\CG^*_{r-k}} \RG{}(\bar{Y}, \Sh{V}^k(k-r)) \to \MF(\CG^*_{r-k})[1] \xrightarrow{+1}.   
        \end{equation}
    \end{enumerate} 
\end{prop}

\begin{proof}
    We first address (1). The fact that there is a unique such map follows from the fact that $\GL_2 \subset \GL_2 \times \GL_2$ (diagonally embedded) is a spherical pair which ensures that any irreducible $\GL_2$ representation occurs with multiplicity at most 1 inside an irreducible $\GL_2 \times \GL_2$ representation. Dualising the map $\det{\!}^r$ and using the fact that $(V^k)^* \simeq V^k(k)$, we get the required map $\CG_{r-k} \colon V^l \otimes V^m \to V^k(k-r)$. Letting $\Delta_2 \colon Y \to Y^2$ be the diagonal embedding, and noting that $\Delta_2^*(\Sh{V}^{[l,m]}) = \Sh{V}^l \otimes \Sh{V}^m$, we define (the second arrow being the natural map from compactly supported cohomology to cohomology without support conditions):
    \[
    \CG_{r-k}^* := \CG_{r-k} \circ \Delta_2^* \colon \RG{c}(\bar{Y}^2, \Sh{V}^{[l,m]}) \xrightarrow{} \RG{c}(\bar{Y}, \Sh{V}^k(k-r))\to \RG{}(\bar{Y}, \Sh{V}^k(k-r)).
    \]
    
    Before proving (2) we first make the following remark. Consider the map $\Delta_*$ as in (\ref{E08}). Dualising the triangle, we get another triangle
    \[
    \RG{c}(\bar{Y}^3, \Sh{V}^{[k,l,m]}) \xrightarrow{\Delta^*} \RG{c}(\bar{Y}, \Q_p(-r)) \to \MF(\Delta^*)[1] \xrightarrow{+1}
    \]
    Taking derived tensor product with $\RG{}(\bar{Y}, \Sh{V}^k(k+1))$ we get another distinguished triangle completing the map 
    \[
    \RG{c}(\bar{Y}^3, \Sh{V}^{[k,l,m]})\otimes^{\mathbf{L}} \RG{}(\bar{Y}, \Sh{V}^k(k+1))\xrightarrow{\Delta^*\otimes^{\mathbf{L}}\id} \RG{c}(\bar{Y}, \Q_p(-r)) \otimes^{\mathbf{L}}\RG{}(\bar{Y}, \Sh{V}^k(k+1))
    \]
    Using the K\"unneth formula, the term on the left decomposes as 
    \[
    \RG{c}(\bar{Y}^3, \Sh{V}^{[k,l,m]})\otimes^{\mathbf{L}} \RG{}(\bar{Y}, \Sh{V}^k(k+1)) = \RG{c}(\bar{Y}^2, \Sh{V}^{[l,m]}) \otimes^{\mathbf{L}} \RG{c}(\bar{Y}, \Sh{V}^k) \otimes^{\mathbf{L}}\RG{}(\bar{Y}, \Sh{V}^k(k+1)). 
    \]
    The pushforward of the natural map $\Q_p \to \Sh{V}^k \otimes \Sh{V}^k(k)$ along the diagonal $\Delta_2 \colon Y \to Y^2$ induces a map
    \[
    \Q_p[-2] \to \RG{}(\bar{Y}, \Q_p)[-2] \to \RG{c}(\bar{Y}, \Sh{V}^k) \otimes^{\mathbf{L}}\RG{}(\bar{Y}, \Sh{V}^k(k+1))
    \]
    which is the dual of the Poincar\'e duality pairing. 

    On the other hand, the fundamental class gives us a map
    \[
    \RG{c}(\bar{Y}, \Q_p(-r)) \to \Q_p(-r-1)[-2].
    \]
    Hence we get a map 
    \begin{equation}\label{E11}
    \RG{c}(\bar{Y}^2, \Sh{V}^{[l,m]})[-2] \xrightarrow{} \RG{}(\bar{Y}, \Sh{V}^k(k-r))[-2]
    \end{equation}
    that fits in the following diagram
    \[\begin{tikzcd}
	{\RG{c}(\bar{Y}^2, \Sh{V}^{[l,m]}) \otimes^{\mathbf{L}} \RG{c}(\bar{Y}, \Sh{V}^k) \otimes^{\mathbf{L}}\RG{}(\bar{Y}, \Sh{V}^k(k+1))} & {\RG{c}(\bar{Y}, \Q_p(-r)) \otimes^{\mathbf{L}}\RG{}(\bar{Y}, \Sh{V}^k(k+1))} \\
	{\RG{c}(\bar{Y}^2, \Sh{V}^{[l,m]}) \otimes^{\mathbf{L}}\Q_p[-2]} & {\Q_p[-2]\otimes^{\mathbf{L}}\RG{}(\bar{Y}, \Sh{V}^k(k-r))}
	\arrow[from=1-1, to=1-2, "{\Delta^*\otimes \id}"]
	\arrow[from=1-2, to=2-2]
	\arrow[from=2-1, to=1-1]
	\arrow[from=2-1, to=2-2]
\end{tikzcd}\]
A simple bookkeeping argument using K\"unneth formula shows that the long exact sequence associated with the distinguished triangle obtained from the map (\ref{E11}) gives rise to the exact sequence (\ref{E10}). Indeed, to obtain (\ref{E10}) from (\ref{E09}) we can first dualize, then tensor with $V^*(f)(-r-1)$, and pullback along the natural map $\Q_p \to V^*(f) \otimes V(f)$. The diagram above is just a derived version of this process which shows that (\ref{E10}) appears in the long exact sequence associated with the triangle completing (\ref{E11}).

Therefore in order to prove (2), we only need to show that the map in (\ref{E11}) coincides with $\CG_{r-k}^*[-2]$. We claim that this follows by applying proper base change to the following Cartesian diagram.
\[\begin{tikzcd}
	Y & {Y \times Y^2} \\
	{Y \times Y} & {Y \times Y^3 = Y^2 \times Y^2}
	\arrow["p_2", from=1-1, to=1-2]
	\arrow["p_1", from=1-1, to=2-1]
	\arrow["{\Delta_2 \times \id}", from=1-2, to=2-2]
	\arrow["{\id \times \Delta}"', from=2-1, to=2-2]
\end{tikzcd}\]
For simplicity of notation, we write $q = \Delta_2 \times \id$, and $p = \id \times \Delta$. Then the claim we want to prove is that 
\[
p^*q_!(\Q_p \boxtimes \Sh{V}^{[l,m]})[-2] \to p^*(\Sh{V}^{[k,k]}(k+1)\boxtimes \Sh{V}^{[l,m]}) = p^*(\Sh{V}^k(k+1)\boxtimes \Sh{V}^{[k,l,m]}) \to \Sh{V}^k(k+1)\boxtimes \Q_p(-r)
\]
coincides with 
\[{p_1}_!(\Sh{V}^l\otimes\Sh{V}^m)[-2] \to {p_1}_!\Sh{V}^k(k-r)[-2] = {p_1}_!p_1^!(\Sh{V}^k(k+1)\boxtimes\Q_p(-r)) \to \Sh{V}^k(k+1)\boxtimes \Q_p(-r).\]
But this follows easily from the relations ${p_1}_!p_2^* = p^*q_!$ and $p_2^*q^! = p_1^!p^*$. This completes the proof.
\end{proof}

\subsection{On comparison isomorphisms}
\label{section5}
In this section we base change our entire setup from $\Spec{\Q}$ to $\Spec{\Q_p}$. In particular, $\bar{Y}$ will now denote $Y_{\bar{\Q}_p}$.

Consider the ``trivial filtration" functor $D^+_{\proet}(\Spec{\Q_p}, \Q_p) \to DF^+_{\proet}(\Spec{\Q_p}, \Q_p)$ sending a complex $C^{\bullet}$ to the filtered complex satisfying $F^{i}(C^{\bullet}) = C^{\bullet}$ if $i \leq 0$, and 0 otherwise. This is an exact functor. Therefore we can view the distinguished triangle \ref{disttriang-diagonalpullback} as a distinguished triangle in $DF^+_{\proet}(\Spec{\Q_p}, \Q_p)$ after base changing everything from $\Spec{\Q}$ to $\Spec{\Q_p}$.

\begin{lemma}\label{de rham comparison}
    The continuous $\Gal(\bar{\Q}_p/\Q_p)$-representations $H^{\bullet}(\MF{(\CG^*_{r-k})})$ are de Rham.
\end{lemma}

\begin{proof}
    Consider the de Rham mapping fibre $\MF_{\dR}(\CG^*_{r-k}) \in DF^+(\Q_p)$ in the filtered derived category of $\Q_p$-vector spaces associated with the pullback $\RG{\dR, c}({Y}^2, \Sh{V}^{[l,m]}) \to \RG{\dR}({Y}, \Sh{V}^k(k-r))$ of de Rham complexes endowed with Hodge filtration. Via the ``constant sheaf" functor followed by tensoring with $B_{\dR}$ we get a distinguished triangle in $DF^+_{\proet}(\Spec{\Q_p}, B_{\dR})$ as follows.
    \[
    \RG{\dR,c}({Y}^2, \Sh{V}^{[l,m]})\otimes^{\mathbf{L}}B_{\dR} \xrightarrow{\CG^*_{r-k}} \RG{\dR}({Y}, \Sh{V}^k(k-r)) \otimes^{\mathbf{L}} B_{\dR} \to \MF_{\dR}(\CG^*_{r-k})[1] \otimes^{\mathbf{L}} B_{\dR} \xrightarrow{+1}.
    \]
    The lemma then follows from the main results of \cite{DLLZ}, \cite{LLZ} which proves the comparison isomorphisms 
    \begin{align*}
    \RG{\proet,c}(\bar{Y}^2, \Sh{V}^{[l,m]}) \otimes^{\mathbf{L}} B_{\dR} &\simeq \RG{\dR,c}({Y}^2, \Sh{V}^{[l,m]}) \otimes^{\mathbf{L}} B_{\dR}, \\
    \RG{\proet}({Y}, \Sh{V}^{k}(k-r)) \otimes^{\mathbf{L}} B_{\dR} &\simeq \RG{\dR}({Y}, \Sh{V}^{k}(k-r)) \otimes^{\mathbf{L}} B_{\dR}
    \end{align*}
    in $DF^+_{\proet}(\Spec{\Q_p}, B_{\dR})$.
\end{proof}

\begin{rem}
Note that the above lemma, together with the isomorphism \eqref{extiso}, allows to (re)prove that the diagonal class $\kappa(f,g,h)$ is de Rham at $p$, i.e., its restriction at $p$ lies in the Bloch--Kato subspace $H^1_g(\Q_p, V^*_{fgh}(-r-1))$. In fact, in our setting -- where $f,g,h$ are crystalline at $p$ -- one checks directly that $H^1_e=H^1_f=H^1_g$ for the three Bloch--Kato conditions for the representation $V^*_{fgh}(-r-1)$ (see \cite[Lemma 3.5]{BSV_reciprocity_balanced}), so $\kappa(f,g,h)$ is in fact crystalline at $p$ (and correspondingly that the extension $\tilde{E}_{\kappa(f,g,h)}$ is crystalline).
\end{rem}

\section{Regulator formula for diagonal classes}
\label{S5}
\subsection{Statement of the $p$-adic regulator formula}
In this section we formulate the statement of the regulator formula for diagonal classes and include its new proof via the formalism developed in the previous sections.

Let $f,g,h$ be three normalized eigenforms of level $\Gamma_1(N)$ and respective weights $k+2,l+2,m+2$. As in Section \ref{subsection: pullback diagonal classes}, we assume that $(k,l,m)$ is a balanced triple of natural numbers and that the product of the nebentypes of $f,g,h$ is trivial, which implies that $r=(k+l+m)/2$ is an integer.

We will also assume that $f$ is new of level $N_1$ for some $N_1\mid N$.

We fix a finite extension $L$ of $\Q_p$ containing all the Fourier coefficients of $f,g,h$ and a primitive $N$-th root of unity. It will be our field of coefficients and for all cohomology groups and complexes appearing below we will always implicitly extend scalars to $L$. From now on in this section we write $Y=Y_1(N)_{\Q_p}$, $X=X_1(N)_{\Q_p}$.

For an eigenform $\xi\in S_\nu(N,\chi_{\xi},L)$ (with $\nu\geq 2$) and $?\in\{*,\emptyset\}$, we set
\[
V^?_{\dR,N}(\xi):=D_{\dR}(V^?_{N}(\xi))=H^0(\Q_p,B_{\dR}\otimes_L V^?_N(\xi))
\]
and we simply write $V^?_{\dR}(\xi)$ if the level $N$ is understood.

Faltings--Tsuji comparison isomorphism yields canonical isomorphisms
\begin{equation}
\label{fil0fil1}
    \Fil^0V^*_{\dR}(\xi)\cong S_\nu(\Gamma_1(N),L)[\xi^w]\qquad \Fil^1 V_{\dR}(\xi)\cong S_\nu(\Gamma_1(N),L)[\xi]
\end{equation}
(where $\xi^w:=w_{N}(\xi)$ for $w_N$ the Atkin--Lehner involution) and a perfect duality
\begin{equation}
\label{dRduality}
\langle -,-\rangle_\xi: V_{\dR}(\xi)\otimes_L V^*_{\dR}(\xi)\to D_{\dR}(L)=L,
\end{equation}
under which we get identifications
\begin{equation}
\label{vdrsufil1}
    V_{\dR}(\xi)/\Fil^1 V_{\dR}(\xi)\cong (S_\nu(\Gamma_1(N),L)[\xi^w])^*
\end{equation}
and
\begin{equation}
    V^*_{\dR}(\xi)/\Fil^0 V^*_{\dR}(\xi)\cong (S_\nu(\Gamma_1(N),L)[\xi])^*,
\end{equation}
where $(-)^\ast$ denotes the $L$-dual of an $L$-vector space.

We let $\omega_g\in\Fil^1 V_{\dR, N}(g)$ be the element corresponding to $g$ under the isomorphism \eqref{fil0fil1}. We will use the same notation $\omega_g$ to denote the holomorphic differential in $H^0(X,\Sh{V}^l\otimes\Omega_{X}(\log D))$ (or in the minus log variant) corresponding to $g$. The same applies to $h$.

We set our convention for the Petersson inner product on the spaces $S_\nu(N,\chi)$ of complex modular forms of level $N$ and character $\chi$ to be
\[
\langle \xi_1, \xi_2\rangle_{\mathrm{Pet}} :=\frac{1}{\mathrm{Vol}(\Sh{H}/\Gamma_0(N))}\int_{\Sh{D}_0(N)}\xi_1(\tau)\overline{\xi_2(\tau)}y^\nu\frac{dxdy}{y^2}
\]
for $\xi_1, \xi_2\in S_\nu(N,\chi)$, where we write $\tau=x+iy\in\Sh{H}$ (the upper half-plane) and $\Sh{D}_0(N)$ is a fundamental domain for the action of $\Gamma_0(N)$ on $\Sh{H}$. In this way, the Petersson product does not depend on the level considered.

We let $\eta_{f}\in V_{\dR,N}(f)/\Fil^1 V_{\dR,N}(f)$ be the element corresponding under the isomorphism \eqref{vdrsufil1} to the linear functional
\begin{equation}
S_{k+2}(\Gamma_1(N),L)[f^w]\to L\qquad \gamma\mapsto\frac{\langle \Tr_{N/N_1}(\gamma), f^w\rangle_{\mathrm{Pet}}}{\langle f,f\rangle_{\mathrm{Pet}}}=[\Gamma_0(N_1):\Gamma_0(N)]\cdot \frac{\langle \gamma, f^w\rangle_{\mathrm{Pet}}}{\langle f,f\rangle_{\mathrm{Pet}}} 
\end{equation}

\begin{rem}
The fact the this linear functional actually takes values in $L$ follows from the work of Hida (cf. \cite[Proposition 4.5]{Hi1985a}).
\end{rem}

Let $\alpha_f,\beta_f$ be the roots of the $p$-th Hecke polynomial for $f$, or equivalently the eigenvalues of crystalline Frobenius on $V_{\dR}(f)=V_{\cris}(f)$ (and similarly for $g$ and $h$). We assume that $\ord_p(\alpha_f)<k+1$, so that we have a decomposition
\[
V_{\dR}(f)=\Fil^1 V_{\dR}(f)\oplus V_{\dR}(f)^{\varphi=\alpha_f}
\]
and we can define $\eta_f^{\alpha}\in V_{\dR}(f)$ as the unique lift of $\eta_f$ to the $\alpha_f$-eigenspace for $\varphi$,

We thus obtain an element $\eta_f^\alpha\otimes\omega_g\otimes\omega_h\otimes e_{r+2}\in\Fil^0 D_{\dR}(V_{fgh}(r+2))$, where $V_{fgh}=V(f)\otimes V(g)\otimes V(h)$ and $e_{r+2}$ is the generator of $D_{\dR}(\Q_p(r+2))$ defined as in Section \ref{section: onBKlog}.

As in \cite[\S 1.1]{BSV2020a}, we can obtain a $p$-adic modular form of weight $k$ out of $g$ and $h$ by the formula
\[
\Xi_k(g,h)=d^{(k-l-m-2)/2}(g^{[p]})\times h,
\]
where $d=q\tfrac{d}{dq}$ is Serre's derivative operator on $p$-adic modular forms and $g^{[p]}$ is the $p$-depletion of $g$ (it is necessary to pass to the $p$-depletion to be able to apply negative powers of $d$ defined via a $p$-adic limit process). We let $f_\alpha$ denote the $p$-stabilization of $f$ with $U_p$-eigenvalue $\alpha_f$ and $f^w_{\alpha}$ be the $p$-stabilization of $f^w$ with $U_p$-eigenvalue $\chi_f(p)^{-1}\alpha_f$. As explained in \cite[\S 1.1]{BSV2020a}, one has an overconvergent $f^w_\alpha$-isotypic projectior $e_{f^w_\alpha}$ on nearly overconvergent modular forms. In fact, we will (re)prove that $\Xi_k(g,h)$ is nearly overconvergent, so that it makes sense to apply $e_{f^w_\alpha}$ to it.

We can finally state the $p$-adic regulator formula for diagonal classes.

\begin{theorem}[cf. Theorem A in \cite{BSV2020a} and Theorem 1.3 in \cite{DR_diagonal_1}]
\label{theorem: diagonal regulator}
Under the assumptions made above, the value of
\[
\log_{\BK}(\mathrm{res}_p(\kappa(f,g,h)))(\eta_f^\alpha\otimes\omega_g\otimes\omega_g\otimes e_{r+2})
\]
is given by
\[
\frac{(-1)^k(r-k)!\Big(1-\frac{\beta_f}{\alpha_f}\Big)\Big(1-\frac{\beta_f}{p\alpha_f}\Big)[\Gamma_0(N_1):\Gamma_0(N)]}{\Sh{E}_p(f,g,h)}\cdot \frac{\big\langle e_{f^w_\alpha}(\Xi_k(g,h)),f^w_\alpha\big\rangle_{\mathrm{Pet}}}{\big\langle f^w_\alpha, f^w_\alpha\big\rangle_{\mathrm{Pet}}},
\]
where
\[
\Sh{E}_p(f,g,h)=\Bigg(1-\frac{\beta_f\alpha_g\alpha_h}{p^{r+2}}\Bigg)\Bigg(1-\frac{\beta_f\alpha_g\beta_h}{p^{r+2}}\Bigg)\Bigg(1-\frac{\beta_f\beta_g\alpha_h}{p^{r+2}}\Bigg)\Bigg(1-\frac{\beta_f\beta_g\beta_h}{p^{r+2}}\Bigg).
\]
\end{theorem}

\begin{rem}
In \cite{Hsi2021}, the author pins down an \emph{optimal} construction of the $p$-adic $L$-functions $\sh{L}_p^{\mathbf{f}}(\mathbf{f}, \mathbf{g}, \mathbf{h})$ for a triple of primitive Hida families $(\mathbf{f},\mathbf{g},\mathbf{h})$ of respective tame levels $(N_1,N_2,N_3)$ under suitable assumptions, including the fact that the big Galois representations attached to $\mathbf{f}$ is absolutely irreducible and that $\gcd(N_1,N_2,N_3)$ is squarefree. We refer to Theorem A of loc. cit for the precise statement. Hsieh's construction relies on a suitable choice of \emph{test vectors} $(\breve{\mathbf{f}},\breve{\mathbf{g}},\breve{\mathbf{h}})$, which are Hida families of common level $N:=\mathrm{lcm}(N_1,N_2,N_3)$ with associated primitive families $(\mathbf{f},\mathbf{g},\mathbf{h})$. If one restricts to the $p$-ordinary case and, under Hsieh's assumptions, thinks of our $g$ (resp. $h$) as the specialization of $\breve{\mathbf{g}}$ (resp. $\breve{\mathbf{h}}$) in weight $l+2$ (resp. $m+2$) and assumes that $\breve{\mathbf{f}}=\mathbf{f}$ is the Hida family passing through $f$ in weight $k+2$, then the expression
\[
[\Gamma_0(N_1):\Gamma_0(N)]\cdot \frac{\big\langle e_{f^w_\alpha}(\Xi_k(g,h)),f^w_\alpha\big\rangle_{\mathrm{Pet}}}{\big\langle f^w_\alpha, f^w_\alpha\big\rangle_{\mathrm{Pet}}}
\]
essentially coincides with the evaluation of Hsieh's $\sh{L}_p^{\mathbf{f}}(\mathbf{f}, \mathbf{g},\mathbf{h})$ at the triple $(k+2,l+2,m+2)$ (cf. for instance, \cite[Proposition 3.6 and Remark 3.7]{Mar2026} for a more precise statement). Analogous considerations apply to the finite slope case, since the choice of test vectors only involves the behaviour of the three families at the primes dividing $N$ (hence away from $p$).
\end{rem}

\subsection{From de Rham to rigid cohomology}
\label{rigidsubsection}
The variety $X$ admits a smooth model $\Sh{X}$ defined over $\Z_p$, so it makes sense to consider the rigid cohomology of the special fibers of such models (always with log poles at the cusps), which admits canonical isomorphism to the de Rham cohomology of $X$. The same applies to $X^2$, with model $\Sh{X}^2$. 

\begin{lemma}
\label{lemma-rigid ext}
    The filtered Frobenius module $D_{\cris}(\tilde{E}_{k(f,g,h)})$ is realised as a Frobenius module in the mapping fibre of the Clebsch--Gordan pullback in rigid cohomology, with the filtration induced from the de Rham mapping fibre via the comparison between de Rham and rigid cohomology.
\end{lemma}

\begin{proof}
We have an isomorphism of short exact sequence of vector spaces as follows, where the top row is canonically equipped with a filtration, $\MF_{\rig}(\CG^*_{r-k})$ is the mapping fibre of the Clebsch--Gordan pullback of the rigid cohomology complexes, and the middle vertical arrow is the natural map induced on mapping fibres by the two specialisation maps from algebraic de Rham cohomology to rigid cohomology. (We choose \v{C}ech resolutions to work with concrete complexes such that the specialisation maps commute with the Clebsch--Gordan pullbacks, inducing the map on mapping fibres.)
\[\begin{tikzcd}
	0 & {H^1_{\dR}(Y,\Sh{V}^k(k-r))} & {H^2(\MF_{\dR}(\CG^*_{r-k}))} & {H^2_{\dR,c}(Y^2,\Sh{V}^{[l,m]})} & 0 \\
	0 & {H^1_{\rig}(X, \Sh{V}^k(k-r)\langle D_X\rangle)} & {H^2(\MF_{\rig}(\CG^*_{r-k}))} & {H^2_{\rig}(X^2, \Sh{V}^{[l,m]}\langle -D_{X^2}\rangle)} & 0
	\arrow[from=1-1, to=1-2]
	\arrow[from=1-2, to=1-3]
	\arrow["\simeq", from=1-2, to=2-2]
	\arrow[from=1-3, to=1-4]
	\arrow["\simeq", from=1-3, to=2-3]
	\arrow[from=1-4, to=1-5]
	\arrow["\simeq", from=1-4, to=2-4]
	\arrow[from=2-1, to=2-2]
	\arrow[from=2-2, to=2-3]
	\arrow[from=2-3, to=2-4]
	\arrow[from=2-4, to=2-5]
\end{tikzcd}\]

We now further restrict to the ordinary locus. As the rigid analytic ordinary locus $X^{\ord}$ admits a global lift of Frobenius, any natural representative of $\RG{\rig}(X^{\ord,\dagger}, \Sh{V}^k(k-r)\langle D_X \rangle)$ -- for example, a \v{C}ech resolution of the overconvergent de Rham complex with respect to a \v{C}ech cover given by tubes of a cover of $X^{\ord}_{\F{p}}$ -- is equipped with a canonical Frobenius. The same is true of $X^{\ord,2}$, and the Clebsch--Gordan pullback on such natural complexes commutes with the Frobenius. Hence, the mapping fibre 
\[
\MF_{\rig}(\CG^*_{r-k, \ord}) := \MF\left[ \RG{\rig}(X^{\ord, 2, \dagger}, \Sh{V}^{[l,m]}\langle -D_{X^2}\rangle) \to \RG{\rig}(X^{\ord, \dagger}, \Sh{V}^k(k-r)\langle D_X \rangle) \right]
\]
is also a mapping fibre in the category of Frobenius modules.

Taking the natural restriction map $\MF_{\rig}(\CG^*_{r-k}) \to \MF_{\rig}(\CG^*_{r-k,\ord})$, we get the following diagram.
\[\begin{tikzcd}
	0 & {H^1_{\dR}(Y,\Sh{V}^k(k-r))} & {H^2(\MF_{\dR}(\CG^*_{r-k}))} & {H^2_{\dR,c}(Y^2,\Sh{V}^{[l,m]})} & 0 \\
	0 & {H^1_{\rig}(X^{\ord,\dagger}, \Sh{V}^k(k-r)\langle D_X\rangle)} & {H^2(\MF_{\rig}(\CG^*_{r-k,\ord}))} & {H^2_{\rig}(X^{\ord,2,\dagger}, \Sh{V}^{[l,m]}\langle -D_{X^2}\rangle)} & 0
	\arrow[from=1-1, to=1-2]
	\arrow[from=1-2, to=1-3]
	\arrow[from=1-2, to=2-2]
	\arrow[from=1-3, to=1-4]
	\arrow[from=1-3, to=2-3]
	\arrow[from=1-4, to=1-5]
	\arrow[from=1-4, to=2-4]
	\arrow[from=2-1, to=2-2]
	\arrow[from=2-2, to=2-3]
	\arrow[from=2-3, to=2-4]
	\arrow[from=2-4, to=2-5]
\end{tikzcd}\]

Moreover, taking colimit over \v{C}ech covers, we may assume that the complexes $\RG{\dR}(Y, \Sh{V}^k(k-r))$ and \hfill \\
$\RG{\rig}(X^{\ord, \dagger}, \Sh{V}^k(k-r)\langle D_X \rangle)$ are equipped with Hecke actions, so that we can localise them at the Hecke eigensystem $\Pi^*_f$ defining $V_{\dR}^*(f)$. In doing so, we can replace the left column of the above diagram with $D_{\dR}(V^*(f)(-r-1)) \simeq D_{\cris}(V^*(f)(-r-1))$, and modify the mapping fibres accordingly, i.e. we take 
\[\MF\left[\RG{\dR,c}(Y^2, \Sh{V}^{[l,m]}) \to \RG{\dR}(Y, \Sh{V}^k(k-r))_{\Pi^*_f}\right]\] and similarly for the rigid complex. 

On the right hand column, the maps 
\[
H^2_{\dR,c}(Y^2, \Sh{V}^{[l,m]}) \simeq H^2_{\rig}(X^2, \Sh{V}^{[l,m]}\langle -D_{X^2} \rangle) \to H^2_{\rig}(X^{\ord, 2,\dagger}, \Sh{V}^{[l,m]}\langle -D_{X^2}\rangle)
\]
induce isomorphisms on the generalised Hecke eigenspace of $V_{\dR}(g) \otimes V_{\dR}(h)$. The last two rigid cohomology groups carry a Frobenius that commutes with the Hecke actions. Moreover, by the results of \cite{ABSV} or by Liebermann's trick, the rigid Frobenius on $H^2_{\rig}(X^2, \Sh{V}^{[l,m]}\langle -D_X^2\rangle)$ matches the crystalline Frobenius under the isomorphism $D_{\cris}(H^2_{\et,c}(Y^2, \Sh{V}^{[l,m]}) \simeq H^2_{\rig}(X^2, \Sh{V}^{[l,m]}\langle -D_{X^2}\rangle)$. As a result, pulling back via $V_{\dR}(g) \otimes V_{\dR}(h) \to H^2_{\dR,c}(Y^2, \Sh{V}^{[l,m]})$ we obtain an extension 
\[
0 \to D_{\dR}(V^*(f)(-r-1)) \to D(E) \to V_{\dR}(g) \otimes V_{\dR}(h) \to 0
\]
where $D_{\dR}(\tilde{E}_{k(f,g,h)}) \simeq D(E)$ as filtered $\Q_p$-vector spaces via the de Rham comparison of Lemma \ref{de rham comparison}. Moreover, $D(E)$ is also an extension in the category of Frobenius modules, and since $P_{gh}(\varphi)$ is invertible on $D_{\cris}(V^*(f)(-r-1))$, it is the unique such extension. In other words, $D_{\cris}(\tilde{E}_{k(f,g,h)}) \simeq D(E)$ as a filtered Frobenius module.
\end{proof}

\begin{rem}\label{rem 6.7}
    Using Besser's ``canonical" rigid complexes, which are equipped with a Frobenius, and such that the Frobenii commute with pullback, we could treat $\MF_{\rig}(\CG^*_{r-k})$ (and its $\Pi^*_f$ localised version) as a mapping fibre in the category of Frobenius modules. Note however that we don't claim that $D_{\cris}(H^2(\MF_{\proet}(\CG^*_{r-k})) \simeq H^2(\MF_{\rig}(\CG^*_{r-k}))$, although this would follow from the same argument as above after noticing that the Frobenius weights of $H^2_{\rig}(X^2, \Sh{V}^{[l,m]}\langle -D_{X^2}\rangle)$ and $H^1_{\rig}(X, \Sh{V}^k(k-r)\langle D_X \rangle)$ are distinct, and by the de Rham comparison of Lemma \ref{de rham comparison}. Alternatively, one can deduce this from the crystalline (in fact, semi-stable) comparison results of \cite{ABSV}.
\end{rem}

\subsection{Proof of the $p$-adic regulator formula}
The rest of this section is devoted to an alternative proof of Theorem \ref{theorem: diagonal regulator} which relies on the results of the previous sections.

Let $P_g(T)=\Big(1-\frac{T}{\alpha_g}\Big)\Big(1-\frac{T}{\beta_g}\Big)\in 1+T\cdot L[T]$ and define $P_h(T)$ analogously. Set 
\[
P_{gh}(T):= (P_g*P_h)(T)=\Big(1-\tfrac{T}{\alpha_g\alpha_h}\Big)\Big(1-\tfrac{T}{\alpha_g\beta_h}\Big)\Big(1-\tfrac{T}{\beta_g\alpha_h}\Big)\Big(1-\tfrac{T}{\beta_g\beta_h}\Big).
\]

The first step is to combine Proposition \ref{proposition: diagonal pullback}, Proposition \ref{Prop: BKlog2}, and Lemma \ref{lemma-rigid ext} to obtain the following corollary.

\begin{cor}
\label{corollary: BK1 formula diagonal}
Consider the extension \eqref{E10} and let $\tilde{\omega}\in D_{\dR}(\tilde{E}_{k(f,g,h)})$ be any preimage of 
\[
\omega:=\omega_g\otimes\omega_h\in\Fil^{l+m+2}D_{\dR}(V(g)\otimes V(h))
\]
Then, with this choice of $P=P_{gh}$ and $\omega$, the hypothesis of Proposition \ref{Prop: BKlog2} are satisfied and we have
\[
ev_\omega\big(\log_{\BK}(\mathrm{res}_p(\kappa(f,g,h)))\big)=\tilde{\omega}-P_{gh}(\varphi)^{-1}\big(P_{gh}(\varphi)\tilde{\omega}\big)
\]
as elements of $D_{\dR}(V^*(f)(-r-1))$.
\end{cor}

\begin{rem}
Note that the equality stated above holds already in $D_{\dR}(V^*(f)(-r-1))$, since one can easily check that $\Fil^{l+m+2-r-1}V_{\dR}^*(f)=(0)$, thanks to the balanced condition on $(k,l,m)$.
\end{rem}

The idea is now to make an explicit choice for $\tilde{\omega}$. In what follows, we write $\omega_{\dR}$ when we mean the class in de Rham cohomology and $\omega_{\coh}$ when we mean the corresponding section in 
\[
H^0(X^2,p_1^*\Sh{V}^l\otimes p_2^*\Sh{V}^m \otimes\Omega_{X^2}\langle -D_{X^2}\rangle)
\]
where (from now on in this section) $p_i\colon X^2\to X$ is the projection on the $i$-th component and $D_{X^2}=p_1^{-1}(D_X)\cup p_2^{-1}(D_X)$ is the boundary divisor for $Y^2\subset X^2$. 

Note that for any chosen \v{C}ech covering, $\omega_{\coh}$ defines a 2-cocycle in the \v{C}ech complex $\RG{\dR}(Y^2,\Sh{V}^{[l,m]})$ giving rise to the class $\omega_{\dR}$, and being a 2-form $\Delta_2^*(\omega_{\coh})=0$ already at the level of cochains. The same applies to the de Rham version of the map $\CG_{r-k}$ of Proposition \ref{proposition: diagonal pullback}. This proves the following lemma.

\begin{lemma}
With the above notation, consider the short exact sequence attached to the de Rham version of the distinguished triangle \eqref{disttriang-diagonalpullback}: 
\[
0 \to H^1_{\dR}(Y,\Sh{V}^k(k-r)) \to H^2(\MF_{\dR}(\CG^*_{r-k})) \to H^2_{\dR,c}(Y^2,\Sh{V}^{[l,m]})\to 0.
\]
Then a distinguished lift of the class $\omega_{\dR}$ to $H^2(\MF_{\dR}(\CG^*_{r-k}))$ is given by the class $[(0,\omega_{\coh})]$. Moreover, the class $[(0,\omega_{\coh})]$ lies in $\Fil^{l+m+2} H^2(\MF_{\dR}(\CG^*_{r-k}))$.
\begin{proof}
We only have to justify the last statement. Since the Clebsch--Gordan pullback on de Rham complexes respects Hodge filtration, the filtration on $\MF_{\dR}(\CG^*_{r-k})$ is necessarily given by the natural map
\[
\MF\left(\Fil^{i}\RG{\dR,c}(Y^2, \Sh{V}^{[l,m]}) \to \Fil^{i}\RG{\dR}(Y, \Sh{V}^k(k-r))\right) \to \MF_{\dR}(\CG^*_{r-k}), \quad i \in \Z.
\]
The claim now follows immediately from the fact that $\omega_{\dR}\in\Fil^{l+m+2}$.
\end{proof}
\end{lemma}

We can thus let $\tilde{\omega}$ to be the class in $\Fil^{l+m+2}D_{\dR}(\tilde{E}_{\kappa(f,g,h)})$ obtained from $[(0,\omega_{\coh})]$. For this choice, the formula of Corollary \ref{corollary: BK1 formula diagonal} simplifies to
\[
ev_\omega\big(\log_{\BK}(\mathrm{res}_p(\kappa(f,g,h)))\big)=-P_{gh}(\varphi)^{-1}\big(P_{gh}(\varphi)\tilde{\omega}\big).
\]

Since by design $P_{gh}(\varphi)(\omega_{\dR})=0$, it follows that there is a unique class
\[
\alpha_{\rig}\in H^1_{\rig}(X,\Sh{V}^k(k-r)\langle D_X \rangle)\]
such that if $\alpha$ is any 1-cocycle in $\RG{\rig}(X,\Sh{V}^k(k-r)\langle D_X \rangle)$ representing $\alpha_{\rig}$, we have $[(\alpha,0)]=P_{gh}(\varphi)[(0,\omega_{\coh})]$. Here by abuse of notation we keep writing $\omega_{\coh}$ for the analytic 2-cocycle in $\RG{\rig}(X^2,\Sh{V}^{[l,m]}\langle -D_{X^2}\rangle )$ given by the image of the algebraic $\omega_{\coh}$ under the natural map from algebraic de Rham complex to analytic de Rham complex. Moreover, we keep the same notation for the coefficients, even if for rigid cohomology the latter are more precisely given by suitable overconvergent Frobenius isocrystals.

As explained in \cite[Section 4.3]{BSV2020a}, the class $\eta_{f,\rig}^\alpha$ (the rigid incarnation of $\eta_f^\alpha$) is the image under extension by zero of a class
\[\eta_{f,\ord}^{\alpha}\in H^1_{\rig,c}(X^{\ord,\dagger},\Sh{V}^k\langle -D_X\rangle)
\]
in the compactly supported rigid cohomology of the ordinary locus. By adjunction, it follows that it is enough to describe the restriction $\alpha_{\ord}$ of $\alpha_{\rig}$ to the ordinary locus, where the complexes computing rigid cohomology become more explicit, as the ordinary locus $X^{\ord,\dagger}$ is affinoid. Moreover, the action of rigid Frobenius is canonically defined already at the level of complexes.

A complex computing $\RG{\rig}(X^{\ord,2,\dagger},\Sh{V}^{[l,m]}\langle -D_{X^2}\rangle)$ can be directly obtained by taking overconvergent (on strict neighbourhoods of the ordinary locus) sections of the complex (here $\nabla_i$ is the connection on the $i$-th factor for $i\in\{1,2\}$)
\[
\Sh{V}^{[l,m]}( -D_{X^2})\xrightarrow{\nabla=\nabla_1\otimes 1+1\otimes\nabla_2}\Sh{V}^{[l,m]}( -D_{X^2})\otimes\Omega^1_{X^2}(\log D_{X^2})\xrightarrow{\nabla}\Sh{V}^{[l,m]}( -D_{X^2})\otimes\Omega^2_{X^2}(\log D_{X^2}).
\]

A similar (and simpler) description applies to $\RG{\rig}(X^{\ord,\dagger},\Sh{V}^k(k-r)\langle D_X\rangle)$ as a two-term complex.

\begin{rem}
\label{rem: ordinaryFrob}
The Clebsch--Gordan pullback over the ordinary loci
\[
CG^*_{r-k, \ord}\colon \RG{\rig}(X^{\ord,2,\dagger},\Sh{V}^{[l,m]}\langle -D_{X^2}\rangle)\to \RG{\rig}(X^{\ord,\dagger},\Sh{V}^k(k-r)\langle D_X\rangle)
\]
is explicitly defined at the level of the complexes described above, in such a way that it commutes with rigid Frobenius at the level of complexes. This implies that a Frobenius induced on $H^i(MF_{\rig}(CG^*_{r-k, \ord}))$ can be allowed to act \emph{component-wise}, i.e., if a class in $\tau\in H^2(MF_{\rig}(CG^*_{r-k, \ord}))$ is represented by $\tau=[(\tau_1,\tau_2)]$, then $\varphi(\tau)$ is represented by $[(\varphi(\tau_1),\varphi(\tau_2))]$. The discussion of \S \ref{rigidsubsection} shows that this \emph{choice} does not affect the outcome of the computation.
\end{rem}

The image of $\tilde{\omega}$ in $H^2(MF_{\rig}(CG^*_{r-k, \ord}))$ can be represented by $[(0,\omega_{\ord})]$, with $\omega_{\ord}$ being the restriction of the differential form $\omega_g\otimes\omega_h$ to the ordinary locus. In particular, we have (cf. Remark \ref{rem: ordinaryFrob}):
\[
[(\alpha_{\ord},0)]=P(\varphi)[(0,\omega_{\ord})]=[(0,P(\varphi)(\omega_{\ord}))].
\]
Looking at how the differential is defined on the mapping fiber, this means that there exist
\[
\beta\in\Gamma\big(X^{\ord,2,\dagger},\Sh{V}^{[l,m]}(-D_{X^2})\otimes\Omega^1_{X^2}(\log D_{X^2})\big),\qquad \gamma\in\Gamma\big(X^{\ord,\dagger},\Sh{V}^k(k-r)\big)
\]
such that
\begin{equation}
\label{eq: key}
P_{gh}(\varphi)(\omega_{\ord})=\nabla(\beta),\qquad \alpha_{\ord}=\nabla(\gamma)-\CG^*_{r-k}(\beta).
\end{equation}

The following chain of equalities shows that we are reduced to computing a pairing in rigid cohomology of the ordinary locus:
\begin{align*}
\log_{\BK}(\mathrm{res}_p(\kappa(f,g,h))&(\eta_f^{\alpha}\otimes\omega_g\otimes\omega_h\otimes e_{r+2})=\\
&=-\big\langle \eta_f^{\alpha}\otimes e_{r+2}\,,\,P_{gh}(\varphi)^{-1}(\alpha_{\rig}))\big\rangle_{\rig, Y}=\\
&=-\big\langle \eta_{f,\ord}^{\alpha}\otimes e_{r+2}\,,\,P_{gh}(\varphi)^{-1}(\alpha_{\ord})\big\rangle_{\rig, X^{\ord,\dagger}}=\\
&=\big\langle \eta_{f,\ord}^{\alpha}\otimes e_{r+2}\,,\,P_{gh}(\varphi)^{-1}([\CG^*_{r-k}(\beta)])\big\rangle_{\rig, X^{\ord,\dagger}}=\\
&=\frac{1}{\mathcal{E}_p(f,g,h)}\big\langle \eta_{f,\ord}^{\alpha}\otimes e_{r+2}\,,\,[\CG^*_{r-k}(\beta)]\big\rangle_{\rig, X^{\ord,\dagger}}.
\end{align*}

The first two equalities should be already clear. The third is a direct consequence of \eqref{eq: key}. For the last, note that $\mathcal{E}_p(f,g,h)$ coincides with the evaluation of $P_{gh}(p^{-1}T^{-1})$ at $T=\alpha_f/p^{r+2}$.

\begin{rem}
Looking at $X^2$ as rigid space, it is clear that the rigid Frobenius $\varphi$ on the rigid cohomology for this product arises as the product of partial Frobenii $\varphi_1$ and $\varphi_2$ on the two components. The same applies when restricting to two copies of the ordinary locus.
\end{rem}

We are thus looking for an explicit overconvergent primitive under $\nabla$ of $P_{gh}(\varphi)(\omega_{g,\ord}\otimes\omega_{h,\ord})$, with $\omega_{g,\ord}$ and $\omega_{h,\ord}$ being the restrictions to $X^{\ord,\dagger}$ of the differentials $\omega_g$ and $\omega_h$. 

As explained in \cite[Section 4.4]{BSV2020a}, $P_g(\varphi_1)(\omega_{g,\ord})=\nabla_1 F_g$ for an explicit overconvergent section $F_g$ (which is unique if $l>0$ and unique up to a constant if $l=0$). Similarly $P_g(\varphi_2)(\omega_{g,\ord})=\nabla_2 F_h$.

Following again loc. cit. (note that our $P_{gh}$ differs slightly from theirs, due to a different choice of normalization of twists), write
\[
P_{gh}(XY)=a(X,Y)\cdot P_g(X)+b(X,Y)\cdot P_h(Y)
\]
with $a(X,Y)=1-\big(\chi_f(p)/p^{l+m+2}\big)X^2Y^2+Y\cdot\, a_0(X,Y)$ and $b(X,Y)=X\cdot\, b_0(X,Y)$ for suitable polynomials $a_0(X,Y),b_0(X,Y)\in L[X,Y]$.

Then, using $\varphi=\varphi_1\cdot \varphi_2$ and the commutativity between connection and Frobenii, we have
\begin{align*}
P_{gh}(\varphi)(\omega_{g,\ord}\otimes\omega_{h,\ord})
= & a(\varphi_1,\varphi_2)(\nabla_1(F_g)\otimes\omega_{h,\ord})+b(\varphi_1,\varphi_2)(\omega_{g,\ord}\otimes\nabla_2(F_h))=\\
= & a(\varphi_1,\varphi_2)\big(\nabla(F_g\otimes\omega_{h,\ord})\big)+b(\varphi_1,\varphi_2)\big(\nabla(\omega_{g,\ord}\otimes F_h)\big)=\\
= & \nabla\big( a(\varphi_1,\varphi_2)(F_g\otimes\omega_{h,\ord})+b(\varphi_1,\varphi_2)(\omega_{g,\ord}\otimes F_h)\big).
\end{align*}
Thus, the section $\beta$ can be described explicitly as 
\[
\beta= a(\varphi_1,\varphi_2)(F_g\otimes\omega_{h,\ord})+b(\varphi_1,\varphi_2)(\omega_{g,\ord}\otimes F_h).
\]
We are left to compute $\CG^*_{r-k,\ord}(\beta)$. 
Arguing exactly in the same way as in loc. cit. (i.e., using the fact that sections with $p$-depleted $q$-expansion are exact over $X^{\ord,\dagger}$), we obtain that, as cohomology classes,
\[
\Delta_2^\ast(\beta)=\big(1-(\chi_f(p)/p^{l+m+2})\varphi^2\big)(F_g\otimes\omega_{h,\ord})\in H^1_{\rig}(X^{\ord,\dagger},\Sh{V}^l\otimes\Sh{V}^m\langle D_X\rangle).
\]
From this point, one can proceed in the proof of Theorem \ref{theorem: diagonal regulator} exactly as explained in \cite[pp. 1023-1024]{BSV2020a}, upon observing that, by definition (see the proof of Proposition \ref{proposition: diagonal pullback}), applying our $\CG_{r-k}$ is the same as taking cup product with the Bertolini--Seveso--Venerucci determinant element $\mathtt{Det}_{\mathbf{r}}$ and applying the dualities
\[
V^\nu\otimes (V^\nu)^\ast\to \Q_p
\]
on coefficients for $\nu\in\{l,m\}$. 

Let us just observe here that, in the formula of Theorem \ref{theorem: diagonal regulator}, the term $(1-\beta_f/p\alpha_f)$) comes from the evaluation of $1-(\chi_f(p)/p^{l+m+2})(pT)^{-2}$ at $T=\alpha_f/p^{r+2}$, while the factor $(1-\beta_f/\alpha_f)$ comes from the $p$-stabilization process.

This concludes our alternative proof of Theorem \ref{theorem: diagonal regulator}.

\section{Connection to syntomic cohomology}
\label{S6}
Consider the following commutative square of maps of chain complexes.
\[\begin{tikzcd}
	{\Fil^{l+m+2}\RG{\dR,c}(Y^2, \Sh{V}^{[l,m]})} & {\RG{\rig}(X^{\ord,2,\dagger}, \Sh{V}^{[l,m]}\langle -D_{X^2}\rangle)} \\
	{\Fil^{l+m+2}\RG{\dR}(Y, \Sh{V}^k(k-r))} & {\RG{\rig}(X^{\ord, \dagger}, \Sh{V}^k(k-r)\langle D_X\rangle)}
	\arrow["{P_{gh}(\varphi)}", from=1-1, to=1-2]
	\arrow["{\CG^*_{r-k}}", from=1-1, to=2-1]
	\arrow["{\CG^*_{r-k,\ord}}", from=1-2, to=2-2]
	\arrow["{P_{gh}(\varphi)}", from=2-1, to=2-2]
\end{tikzcd}\]

Our computations in the previous section can be rephrased as computing the connecting map of snake lemma associated with the following map of short exact sequences.
\[\begin{tikzcd}
	0 & {\Fil^{l+m+2}H^1_{\dR}(Y, \Sh{V}^k(k-r))} & {\Fil^{l+m+2}H^2(\MF_{\dR}(\CG^*_{r-k}))} & {\Fil^{l+m+2}H^2_{\dR}(Y^2, \Sh{V}^{[l,m]})} & 0 \\
	0 & {H^1_{\rig}(X^{\ord,\dagger}, \Sh{V}^k(k-r)\langle D_X\rangle)} & {H^2(\MF_{\rig}(\CG^*_{r-k,\ord}))} & {H^2_{\rig}(X^{\ord, 2, \dagger}, \Sh{V}^{[l,m]}\langle -D_{X^2}\rangle)} & 0
	\arrow[from=1-1, to=1-2]
	\arrow[from=1-2, to=1-3]
	\arrow["{P_{gh}(\varphi)}", from=1-2, to=2-2]
	\arrow[from=1-3, to=1-4]
	\arrow["{P_{gh}(\varphi)}", from=1-3, to=2-3]
	\arrow[from=1-4, to=1-5]
	\arrow["{P_{gh}(\varphi)}", from=1-4, to=2-4]
	\arrow[from=2-1, to=2-2]
	\arrow[from=2-2, to=2-3]
	\arrow[from=2-3, to=2-4]
	\arrow[from=2-4, to=2-5]
\end{tikzcd}\]

\begin{defn}
    Define finite polynomial complexes as
    \begin{align*}
        \RG{P_{gh}-\syn,c}^{\ord}(Y^2, \Sh{V}^{[l,m]}) &:= \MF\left[ \Fil^{l+m+2}\RG{\dR}(Y^2,\Sh{V}^{[l,m]}) \xrightarrow{P_{gh}(\varphi)} \RG{\rig}(X^{\ord,2,\dagger}, \Sh{V}^{[l,m]}\langle -D_{X^2}\rangle )\right] \\
        \RG{P_{gh}-\syn}^{\ord}(Y, \Sh{V}^k(k-r)) &:= \MF\left[\Fil^{l+m+2}\RG{\dR}(Y, \Sh{V}^k(k-r)) \xrightarrow{P_{gh}(\varphi)} \RG{\rig}(X^{\ord,\dagger}, \Sh{V}^k(k-r)\langle D_X\rangle )\right].
    \end{align*}
\end{defn}

\begin{cor}
    $P_{gh}(\varphi)\log_{\BK}(\tilde{E}_{k(f,g,h)})(\omega_g \otimes \omega_h)$ is the image of the class $[(\beta, \omega_g \otimes \omega_h)]$ under the syntomic pullback
    \[
    \CG^*_{r-k,\syn}: \RG{P_{gh}(\varphi)-\syn, c}^{\ord}(Y^2, \Sh{V}^{[l,m]}) \to \RG{P_{gh}(\varphi)-\syn}^{\ord}(Y, \Sh{V}^k(k-r)).
    \]
\end{cor}
\begin{proof}
    Follows from Lemma \ref{L202} (or more precisely Remark \ref{rem2.3}(2)).
\end{proof}

\begin{rem}
    In light of Remark \ref{rem 6.7} we could have also defined syntomic complexes without mentioning $\ord$ anywhere -- such complexes are the more traditional definitions of syntomic complexes. However, since for the purposes of finding concrete primitives one has to anyway restrict to the ordinary locus, the ordinary versions come up naturally.
\end{rem}
\printbibliography
\end{document}